\documentclass{article}
\usepackage{amsmath}
\usepackage{amssymb}
\usepackage{amsfonts}
\usepackage{amscd}
\usepackage{mathrsfs}
\usepackage{amsthm}

\newcommand{\CC}{{\rm\bf C}}
\newcommand{\RR}{{\rm\bf R}}
\newcommand{\QQ}{{\rm\bf Q}}
\newcommand{\ZZ}{{\rm\bf Z}}

\newcommand{\Adeles}{{\rm\bf A}}

\newcommand{\OO}{\mathcal {O}}

\DeclareMathOperator{\Spec}{\mathrm{Spec}}

\DeclareMathOperator{\GL}{\mathrm{GL}}

\DeclareMathOperator{\GSpin}{\mathrm{GSpin}}
\DeclareMathOperator{\Oo}{\mathrm {O}}

\DeclareMathOperator{\SO}{\mathrm {SO}}

\DeclareMathOperator{\Aut}{\mathrm {Aut}}
\DeclareMathOperator{\End}{\mathrm {End}}
\DeclareMathOperator{\Hom}{\mathrm {Hom}}

\DeclareMathOperator{\kernel}{\mathrm {ker}}
\DeclareMathOperator{\image}{\mathrm {im}}

\DeclareMathOperator{\tr}{\mathrm{tr}}
\DeclareMathOperator{\Gal}{\mathrm {Gal}}
\DeclareMathOperator{\ind}{\mathrm{Ind}}

\DeclareMathOperator{\res}{\mathrm{Res}}

\DeclareMathOperator{\Lie}{\mathrm{Lie}}

\DeclareMathOperator{\sgn}{\mathrm{sgn}}
\DeclareMathOperator{\Int}{\mathrm {Int}}

\newcommand{\liea}{{\mathfrak {a}}}
\newcommand{\lieb}{{\mathfrak {b}}}
\newcommand{\liec}{{\mathfrak {c}}}
\newcommand{\lied}{{\mathfrak {d}}}

\newcommand{\lieg}{{\mathfrak {g}}}
\newcommand{\lieh}{{\mathfrak {h}}}

\newcommand{\liek}{{\mathfrak {k}}}
\newcommand{\liel}{{\mathfrak {l}}}
\newcommand{\liem}{{\mathfrak {m}}}
\newcommand{\lien}{{\mathfrak {n}}}

\newcommand{\liep}{{\mathfrak {p}}}
\newcommand{\lieq}{{\mathfrak {q}}}

\newcommand{\lies}{{\mathfrak {s}}}
\newcommand{\liet}{{\mathfrak {t}}}
\newcommand{\lieu}{{\mathfrak {u}}}
\newcommand{\liev}{{\mathfrak {v}}}

\newcommand{\liegk}{{\mathfrak {gk}}}

\theoremstyle{Theorem}
\newtheorem{introconjecture}{Conjecture}

\newtheorem{introtheorem}[introconjecture]{Theorem}

\theoremstyle{plain}
\newtheorem{theorem}{Theorem}[section]
\newtheorem{lemma}[theorem]{Lemma}
\newtheorem{corollary}[theorem]{Corollary}
\newtheorem{proposition}[theorem]{Proposition}
\newtheorem{conjecture}[theorem]{Conjecture}

\theoremstyle{remark}

\begin{document}

\title{On Period Relations for\\Automorphic $L$-functions II}
\author{Fabian Januszewski}
\date{\today}

\maketitle

\begin{abstract}
We study Hecke algebras for pairs $(\lieg,K)$ over arbitrary fields $E$ of characteristic $0$, define the Bernstein functor and give another definition of the Zuckerman functor over $E$. Building on this and the author's previous work on rational structures on automorphic representations, we show that hard duality remains valid over $E$ and apply this result to the study of rationality properties of Sun's cohomologically induced functionals. Our main application are period relations for the special values of standard $L$-functions of automorphic representations of $\GL(2n)$ admitting Shalika models.
\end{abstract}

{
\tableofcontents
}

\section*{Introduction}

In \cite{januszewskipart1} the author used the rational structures constructed in \cite{januszewskipreprint} to prove period relations for Rankin-Selberg $L$-functions. Here we generalize the method to standard $L$-functions of $\GL(2n)$ for cuspidal representations which are lifts from globally generic cuspidal representations of $\GSpin(2n+1)$. In order to prove this, we are naturally led to study Bernstein functors and hard duality over arbitrary fields of characteristic $0$, which is also of independent interest.

Our main result on special values is the following. For each irreducible regular algebraic cuspidal automorphic representations $\Pi$ of $\GL(2n)$ over a totally real field $F$ which is a Langlands lift from a globally generic cuspidal representation of $\GSpin(2n+1)$, subject to a mild condition on the weight, we have for each $k\in\ZZ$ such that $s=\frac{1}{2}+k$ is critical for the standard $L$-function $L(s,\Pi)$ in the sense of Deligne \cite{deligne1979},
\begin{equation}
L(\frac{1}{2}+k,\Pi)\;\in\;(2\pi i)^{m\cdot k}\cdot\QQ(\Pi)\cdot\Omega_{(-1)^k}(\Pi)
\label{eq:introperiodrelation}
\end{equation}
with two complex constants $\Omega_\pm(\Pi)\in\CC^\times$ independent of $k$, $\QQ(\Pi)$ the field of definition of $\Pi$.


To such a $\Pi$ we may conjecturally attach a motive $M(\Pi)$ with the property that
\begin{equation}
L(s-\frac{2n-1}{2},\Pi)\;=\;L(s,M(\Pi)).
\label{eq:pimotivel}
\end{equation}
Assuming the existence of $M(\Pi)$, Deligne's Conjecture \cite{deligne1979} predicts, among many other things, that \eqref{eq:introperiodrelation} is true with $m=nr_F$, where $r_F$ is the degree of $F$ over $\QQ$ (cf.\ section \ref{sec:deligne}).

As is well known, we have an automorphic cohomological definition for periods $\Omega_\pm(\Pi,s_0)\in\CC^\times$ under mild conditions on the \lq{}weight\rq{} of $\Pi$, such that \eqref{eq:introperiodrelation} holds with periods {\em varying} with $k$.

It is known that the periods $\Omega_\pm(\Pi,s_0)\in\CC^\times$ behave as expected under twists with finite order Hecke characters (cf.\ \cite{ashginzburg1994,grobnerraghuram2014}).

The contents of our Theorems A below (Theorem \ref{thm:globalshalikaperiods} in the text) is that the predicted period relations for varying $s_0$ are indeed satisfied. The case $n=1$ coincides with the previously treated Rankin-Selberg theory from \cite{januszewskipart1}, and corresponds to the case of Hilbert modular forms on $\GL(2)/F$, which is well known (cf.\ \cite{manin1972,shimura1976,manin1976,shimura1977,shimura1978}).

\begin{introtheorem}\label{introthm:globalshalikaperiods}
Assume $n\geq 1$, let $F$ be a totally real number field of degree $r_F$ over $\QQ$, and $\Pi$ be a algebraic irreducible cuspidal automorphic representation of $\GL_{2n}(\Adeles_F)$ admitting a Shalika model. Assume that $\Pi_\infty$ has non-trivial relative Lie algebra cohomology with coefficients in an irreducible rational $G_{2n}$-module $M$, which we assume to be critical in the sense of section \ref{sec:artihmeticity}. Then there exist non-zero periods $\Omega_\pm$, numbered by the $2^{[F:\QQ]}$ characters $\pm$ of $\pi_0(F_\infty^\times)$, such that for each critical half integer $s_0=\frac{1}{2}+j_0$ for the standard $L$-function $L(s,\Pi)$, and each finite order Hecke character
$$
\chi:\quad F^\times\backslash\Adeles_F^\times\to\CC^\times,
$$
we have, in accordance with Deligne's Conjecture (cf.\ Conjecture \ref{conj:shalikadeligne}),
$$
\frac{L(s_0,\Pi\otimes\chi)}
{G(\overline{\chi})^{n}(2\pi i)^{j_0r_Fn}
\Omega_{(-1)^{j_0}\sgn\chi}}
\;\in\;
\QQ(\Pi,\chi).
$$
Furthermore the expression on the left hand side is $\Aut(\CC/\QQ)$-equivariant.
\end{introtheorem}

Our proof of Theorem A uses similar methods as applied to the Rankin-Selberg case in \cite{januszewskipart1}. In particular we make use of Theorem 2.3 of loc.\ cit., which allows us to control the field of rationality of the $G(\RR)^0$-representations occuring in $\Pi_\infty$. Once again the remarkable fact that a root system of type $D_n$ admits the negated long Weyl element $-w_0$ as a non-trivial automorphism if and only if $n$ is odd implies that the field of rationality of the irreducible $G(\RR)^0$-subrepresentations of $\Pi_\infty$ matches the contribution of $\sqrt{-1}$ to the period relation \eqref{eq:introperiodrelation}.

However, in the case at hand we need a more refined argument to show that the archimedean zeta integral, evaluated at critical points, produces $\QQ(\Pi)$-{\em rational} functionals on $\Pi_\infty$, considered as a Casselman-Wallach representation.

Even though $(\GL_{2n}(\RR),\GL_n(\RR)\times\GL_n(\RR))$ is known to be a Gelfand pair (cf.\ \cite{aizenbudginzburgsayag2009}), we do not know if it satisfies the strong Gelfand property. Currently we only know that for {\em generic} (quasi-)characters $\chi$ of $H:=\GL_n(\RR)\times\GL_n(\RR)$ we have
$$
\dim\Hom_{H}(\Pi_\infty,\chi)\leq 1.
$$
As a result we presently do not know the corresponding multiplicity one statement on the level of $(\lieg,K)$-modules, even though the automatic continuity theorem from \cite{bandelorme1988,brylinskidelorme1992} we invoke holds for general $\chi$.

In order to overcome this obstacle we consider Sun's cohmologically induced functionals as introduced in \cite{sunpreprint2} in his proof of the non-vanishing of the Shalika periods which are under investigation here. The key idea is to show that for critical $s$ these functionals are defined over the field of rationality of the representation, which we achieve in Theorem \ref{thm:rationalfunctionals}.

The proof of this theorem leads us to the consideration of Hard Duality for Zuckerman and Bernstein functors over arbitrary fields of characteristic $0$. Zuckerman functors over arbitrary fields were already introduced by the author in \cite{januszewskipreprint}, and their $dg$ analogues over arbitrary rings were introduced by Hayashi \cite{takumapreprint1}, who also defined $dg$ analogues of Bernstein functors over $\CC$. However Bernstein functors have not yet been defined over fields other than the complex numbers. In oder to give a unified treatament of all functors and also of the construction of functionals we generalize the appropriate Hecke algebras to arbitrary base fields of characteristic $0$.

In our treatment we are conceptually close to the monograph \cite{book_knappvogan1995} for the simple reason that we need detailed knowledge about the interal structure of the object and in particular we rely on several fundamental compatibility relations in our applications, which in many cases can be readily deduce from the detailed treatment in loc.\ cit. There is no doubt that there are more efficient ways to prove the statement Hard Duality for different definitions of the Bernstein and Zuckerman functors as in loc.\ cit., but in order to establish all the compatibility relations alluded to above, it seems that the total effort would remain the same.

We hope that the representation theoretically not so inclinded reader appreciates our decistion to go along these lines, as it allows us to give a single reference for almost all representation theoretic statements we need.

The paper is divided in two parts. The first two sections treat the general theory of Hecke algebras, Bernstein and Zuckerman functors and Hard Duality over general fields of characteristic $0$. The third section complements this with a conceptual construction and generalization of Sun's cohomologically induced functionals. The remaining five sections treat our application to special values of $L$-functions. These sections rely crucially on the general results from the first three sections.

This work was part of the author's Habilitation thesis. It is the author's great pleasure to thank Claus-G\"unther Schmidt and Stefan K\"uhnlein from Karlsruhe, Don Blasius and Haruzo Hida from UCLA, and Anton Deitmar from T\"ubingen, for being available as members of the author's Habilitation committee. The author also thanks Dragan Milicic and Pavle Pand\v{z}i\'c for very helpful conversations about hard duality.

\section*{Notation}

We let $\overline{\QQ}\subseteq\CC$ denote the algebraic closure of $\QQ$ inside $\CC$. If $v$ is a place of a field $E$, we write $E_v$ for its completion at $v$. If $E/\QQ$ is a number field, we write $\Adeles_E=\Adeles\otimes_\QQ E$ for the topological ring of adeles over $E$, where $\Adeles$ denotes the ring of adeles over $\QQ$. We write $\Adeles_E^{(\infty)}=\Adeles^{(\infty)}\otimes_\QQ E$ for the ring of finite adeles.

If $E/F$ is a finite separable field extension, we write $\res_{E/F}$ for the functor of restriction of scalars \`a la Weil for the extension $E/F$, sending quasi-projective varieties over $E$ to quasi-projective varieties over $F$. This functor preserves finite products and therefore sends group objects to group objects.

If $G$ is an algebraic or a topological group, we denote its connected component of the identity by $G^0$. If $\lieg$ is the (complexified) Lie algebra of a real Lie group $G$, we let $U(\lieg)$ denote its universal enveloping algebra. If $G$ is an algebraic group defined over a field $E$, then $\lieg$ and $U(\lieg)$ are defined over $E$ as well and if we emphasize that we consider these objects over $E$, we write $\lieg_E$ and $U(\lieg_E)$ or also $U_E(\lieg)$.

If $G$ is an algebraic group defined over a field $E$, we always assume that its connected component $G^0\subseteq G$ is always defined over $E$ as well. In other words $G^0$ is a geometrically connected subgroup of $G$, defined over $E$, and of finite index in $G$.

We write
$$
X(G)\;=\;\Hom(G,\GL_1)
$$
for the group of rational characters of $G$. If $G$ is defined over a field $E$, we denote by $X_E(G)$ the subgroup of characters which are defined over $E$.

Over $\RR$, reductive pairs are understood as in \cite{book_knappvogan1995}, i.e.\ a reductive pair $(\lieg,K)$ consits of a pair $(\lieg,K)$ where $\lieg$ is reductive, and furthermore a real form $\lieg_0$ of $\lieg$, a Cartan involution $\theta$ on $\lieg_0$, and a non-degenerate bilinear form $\langle\cdot,\cdot\rangle$ on $\lieg_0$ are given, subject to the usual compatibility conditions (cf.\ Definition  4.30 in loc.\ cit.).

To such a pair we may associate unique reductive Lie group $G$ with Lie algebra $\lieg_0$ containing $K$ as a maximal compact subgroup in such a way that the map $(\lieg,K)\mapsto (G,K)$ sets up an equivalence between (the categories of) classical reductive pairs and real reductive Lie groups.

If $(V,\rho)$ is a Casselman-Wallach representation of a real reductive Lie group $G$, i.e.\ $V$ is a finitely generated, admissible, Fr\'echet representation of moderate growth. If $K\subseteq G$ is a maximal compact subgroup, we write $V^{(K)}$ for the subspace of $K$-finite vectors. This is a finitely generated admissible $(\lieg,K)$-module, and is irreducible if and only if $V$ is (topologically) irreducible. More generally the categories of finite length $(\lieg,K)$-modules is equivalent to the category of Casselman-Wallach representations of $G$.

A superscript $(\cdot)^\vee$ on a (rational / locally $K$-finite) representation of $G$, of $G(\RR)$ or a $(\lieg,K)$-module denotes its (rational / locally $K$-finite) dual.

We assume without loss of generality that all Haar measures on totally disconnected groups we consider have the property that the volumes of compact open subgroups are {\em rational numbers}.

We adopt the convention that for any vector space $V$ the natural map
$$
\bigwedge^p V\otimes \bigwedge^q V\;\to\;
\bigwedge^{p+p}V,
$$
is always uniquely determined by the order of the factors in the tensor product, i.e.\ the above map is given by
$$
v\otimes w\;\mapsto\;v\wedge w.
$$

\section{Hecke Algebras}

In this section we first recall the notion of pair $(\liea,B)$ over an arbitrary base field $F$ of characteristic $0$ and also the fundamental properties of the category of $(\liea,B)$-modules and $(\lieg,K)$-modules over $F$ discussed in \cite{januszewskipreprint}. Then we go on to define Hecke algebras for pairs $(\liea,B)$ over arbitrary base fields of characteristic $0$. Such Hecke algebras were defined for reductive pairs over $\CC$ in the 1970ies by Deligne and Flath. Knapp and Vogan give a systematic and general treatment over $\CC$ in Chapter 1 in their monograph \cite{book_knappvogan1995}. Many of their arguments carry over to the more general context once the Hecke algebras are appropriately defined. In particular the definition and fundamental properties of the Zuckerman and Bernstein functors, as well as their adjoints are straightforward.

\subsection{Rational models of pairs and $(\lieg,K)$-modules}

In this section $G$ denotes a general reductive group over a general base field $F$ of characteristic $0$. We also let $K\subseteq G$ denote a closed reductive subgroup defined over $F$.

We write $\lieg$, $\liek$, ... for the Lie algebras of $G$, $K$, ... respectively. All these Lie algebras are defined over $F$. To stress the base field $E/F$ under consideration, we write
$$
\lieg_E\;:=\;\lieg\otimes_F E,
$$
and similarly
$$
G_E\;:=\;G\times_F E
$$
for the base change of $G\to\Spec E$ to $G_E\to\Spec E$. Then $(\lieg_E,K_E)$ is a reductive pair over $E$ in the sense of \cite{januszewskipreprint}. For the sake of readability we introduce the notation
$$
(\lieg,K)_E\;:=\;(\lieg_E,K_E),
$$
that we also apply to more general pairs. If the base field $E$ is clear from the context, we drop it from the notation.

We give a brief sketch of the theory of Harish-Chandra modules over arbitrary fields of characteristic $0$. For the fundamental theory of modules and cohomological induction over arbitrary fields, as well as the rationality results we need, we refer to \cite{januszewskipreprint}.

A pair over a field $E/\QQ$ consists of an $E$-Lie algebra $\liea_E$ and a reductive algebraic group $B_E$ over $E$, together with an inclusion
$$
\iota_E:\quad \Lie(B_E)\to\liea_E,
$$
of Lie algebras and an action of $B_E$ on $\liea_E$, extending the action of $B_E$ on $\Lie(B_E)$, whose derivative is the adjoint action of $\Lie(B_E)$ on $\liea_E$, the latter action being induced by $\iota_E$.

Then an $(\liea,B)_E$-module $X_E$ consists of an $E$-vector space $X_E$ together with compatible actions of $\liea_E$ and $B_E$. Here we implicitly assume that $X_E$ is a {\em rational} $B_E$-module, which amounts to saying that it is a direct sum of finite-dimensional rational representations of $B_E$. Then this rational action induces an action of $\Lie(B_E)$ on $X_E$, and we assume the action of $\liea_E$ to be an extension of this action. Furthermore we assume the given adjoint action of $B_E$ on $\liea_E$ and the given action on $X_E$ to be compatible with the action of $\liea_E$ in the usual sense.

Then the category $\mathcal C(\liea,B)_E$ of $(\liea,B)_E$-modules (over $E$) is an abelian category, even an $E$-linear tensor category. For each extension of fields $E'/E$ we have natural base change functors
$$
-\otimes_E E':\quad \mathcal C(\liea,B)_E\to \mathcal C(\liea,B)_{E'},
$$
sending an $E$-rational module $X_E$ to the $E'$-rational module
$$
X_{E'}\;:=\;X_E\otimes_E E',
$$
which is an $(\liea,B)_{E'}$-module. A fundamental property of the base change functor is
$$
\Hom_{\liea,B}(X_E,Y_E)\otimes_E E'\;=\;
\Hom_{\liea,B}(X_{E'},Y_{E'})
$$
for all $X_E,Y_E\in\mathcal C(\liea,B)_E$, cf.\ Proposition 1.1 in \cite{januszewskipreprint}.

We write $\widehat{B}_E$ for the set of isomorphism classes of irreducible rational $B$-modules over $E$. If $B$ is connected and $E$ is algebraically closed, this set is parametrized by classical highest weights.

\subsection{Hecke algebras for reductive groups}\label{sec:reductivehecke}

In this section $E$ is any field of characteristic $0$. Let $K$ be a reductive linear group over $E$, not necessarily connected. We let $\OO_E(K)$ denote the coordinate ring of $K$ over $E$. Then $\OO_E(K)$ is a commutative $k$-algebra of finite type, furnished with an additional structure of coalgebra with antipode, stemming from the group structure of $K$. In particular $\OO_E(K)$ is a commutative Hopf algebra. As $K$ is reductive, $\OO_E(K)$ is cosemisimple.

For the elementary definitions and facts about Hopf algebras, Hopf modules, coalgebras, comodules, etc.\ we use, we refer to \cite{book_sweedler1969}.

\subsubsection{Rational representations}

A rational (left) representation $(V,\rho)$ of $K$ over $E$ is the same as a right $H$-comodule structure on $V$. This carries over to not necessarily finite-dimensional $E$-vector spaces $V$: A rational action of $K$ on $E$ is an $\OO_E(K)$-comodule structure on $V$, i.e.\ a $k$-linear map
$$
\rho:\quad V\to V\otimes\OO_E(K)
$$
satisfying the usual properties of coassociativity and counity.

We remark that the coalgebra structure on $\OO_E(K)$ induces an algebra structure on the $E$-linear dual $\OO_E(K)^*$. Then every right $\OO_E(K)$-comodule $V$ is naturally a left $\OO_E(K)^*$-module and the $\OO_E(K)^*$-modules arising this way are called {\em rational}. The category of rational $\OO_E(K)^*$-modules is naturally equivalent to the category of rational $K$-modules.

\subsubsection{Definition of the Hecke algebra}

As $\OO_E(K)$ is cosemisimple as coalgebra, we have a decomposition
\begin{equation}
\OO_E(K)\;=\;\bigoplus_\rho \OO_{\rho,E}
\label{eq:ocodecomposition}
\end{equation}
into simple coalgebras $\OO_{\rho,E}$. We emphasize that this direct sum decomposition is unique.

\begin{proposition}\label{prop:leftrightideals}
For $\lambda\in\OO_E(K)^*$ the following statements are equivalent:
\begin{itemize}
\item[(a)] There exists a left coideal $I\subseteq\OO_E(K)$ of finite codimension with $I\subseteq\ker\lambda$.
\item[(b)] There exists a right coideal $I\subseteq\OO_E(K)$ of finite codimension with $I\subseteq\ker\lambda$.
\item[(c)] There exists a coideal $J\subseteq\OO_E(K)$ of finite codimension with $I\subseteq\ker\lambda$.
\end{itemize}
\end{proposition}

\begin{proof}
This is an easy consequence of the decomposition \eqref{eq:ocodecomposition}. Clerly (c) implies (a) and (b). We show that (a) implies (c). The implication (b) to (c) is proved along the same lines.

Let $I$ be a left coideal of finite codimension. We show that it contains a coideal $J$ of finite codimension. We claim that by the decomposition \eqref{eq:ocodecomposition}, $I$ is the direct sum
\begin{equation}
I\;=\;\bigoplus_{\rho} \left(\OO_{\rho,E}\cap I\right).
\label{eq:idecomposition}
\end{equation}
Indeed, $I$ is a left $\OO_E(K)$-comodule, and decomposes as such into simple comodules
$$
I\;=\;\bigoplus_{\tau}I_\tau.
$$
Each $I_\tau$ is finite-dimensional and for each index $\tau$ there is an index $\rho(\tau)$ with the property that the $\OO_E(K)$-comodule structure on $I_\tau$ is given by the composition
\begin{equation}
I_\tau\to I_\tau\otimes\OO_{\rho(\tau)}\to I_\tau\otimes\OO_E(K).
\label{eq:itaustructure}
\end{equation}
Now since $\OO_{\rho,E}$ are subcoalgebras of $\OO_E(K)$, the coalgebra structure map $\OO_E(K)\to\OO_E(K)\otimes\OO_E(K)$ respects the direct sum decomposition \eqref{eq:ocodecomposition}, and hence \eqref{eq:itaustructure} implies
$$
I_\tau\;\subseteq\;\OO_{\rho(\tau)}.
$$
This proves \eqref{eq:idecomposition}.

Since $I$ is of finite codimension, there is an indexing set $\mathcal R$, which contains all indices $\rho$ except possibly finitely many, with the property that
$$
\forall\rho\in \mathcal R:\quad\OO_{\rho,E}\cap I\;=\;\OO_{\rho,E}.
$$
Then the sum
$$
J\;:=\;\bigoplus_{\rho\in \mathcal R} \OO_{\rho,E}
$$
is a coideal of finite codimension in $\OO_E(K)$ and is contained in $I$.
\end{proof}

Proposition \ref{prop:leftrightideals} is the analogue of the well known fact (cf.\ \cite{book_knappvogan1995}), that for a distribution on a compact Lie group $'K'$ the notions of being left-, right- or bi-$'K'$-finite are all equivalent.

Therefore we all can element $\lambda\in\OO_E(K)^*$ left- or right- or bi-$K$-finite or simply {\em $K$-finite} if it satisfies one of the equivalent conditions in Proposition \ref{prop:leftrightideals}.

We define
$$
R_E(K)\;:=\;\{
\lambda\in\OO_E(K)^*\mid
\lambda\;\text{is $K$-finite}\}.
$$
Then $R_E(K)$ is the maximal rational left-$\OO_E(K)^*$-submodule of $\OO_E(K)^*$, and agrees with the maximal rational right-$\OO_E(K)^*$-submodule. Therefore $R_E(K)$ acquires natural rational left and right $K$-actions.

\subsubsection{Approximate identities}

Each finite-dimesional coalgebra $C$ which occurs as a coalgebra quotient of $\OO_E(K)$ may be naturally identified with a finite indexing set $\mathcal R_C$ with the property that
$$
C\;\cong\;\bigoplus_{\rho\in\mathcal R_C}\OO_{\rho,E},
$$
via the natural inclusion of the right hand side into $\OO_E(K)$, composed with the projection onto $C$. Then to $C$ correponds a subalgebra $R_E(K)_C\subseteq\O_E(K)^*$.

\begin{proposition}\label{prop:finitehecke}
For any coalgebra quotient $C$ of $\OO_E(K)$ we have $R_E(K)_C\subseteq R_E(K)$ and $R_E(K)$ is the inductive limit of all $R_E(K)_C$ in the category of associative $E$-algebras.
\end{proposition}

\begin{proof}
With the notation above, we define the ideal
$$
J_C\;:=\;\ker(\OO_E(K)\to C)\;=\;\bigoplus_{\rho\not\in\mathcal R_C}\OO_{\rho,E}.
$$
It is of finite codimension and we have
$$
R_E(K)_C\;=\;\{\lambda\in\OO_E(K)^*\mid J_C\subseteq\ker\lambda\}.
$$
Therefore $R_E(K)_C\subseteq R_E(K)$ and $R_E(K)$ is the union of all $R_E(K)_C$, since every ideal $J\subseteq\OO_E(K)$ of finite codimension is of the form $J_C$ for some finite-dimensional coalgebra quotient $C$.
\end{proof}

\begin{corollary}\label{cor:heckealgebrastructure}
The algebra structure on $\OO_E(K)^*$ induces on $R_E(K)$ the structure of associative algebra over $E$ with approximate identity.
\end{corollary}

\begin{proof}
As the algebra structure on $\OO_E(K)^*$ is compatible with the algebra structure on each $R_E(K)_C$ for $C$ as in Proposition \ref{prop:finitehecke}, the restriction of the algebra multiplication on $\OO_E(K)^*$ to $R_E(K)$ has image in $R_E(K)$, since the latter is the union of the algebras $R_E(K)_C$. Furthermore the identities in $R_E(K)_C$ form an approximate identity in $R_E(K)$.
\end{proof}

We call a unit $1_C:=1_{R_E(K)_C}\in R_E(K)$ an {\em elementary idempotent}. We remark that
\begin{equation}
R_E(K)_C\;=\;1_CR_E(K)1_C.
\label{eq:rc}
\end{equation}
An $E$-vector space $V$, together with a homomorphism of associative algebras $\rho:R_E(K)\to\End_E(V)$ is called an approximately unital left $R_E(K)$-module, if for every finite-dimensional subspace $W\subseteq V$ there is an elementary idempotent $1_C\in R_E(K)$ with the property $1_CW=W$. Approximately unital right $R_E(K)$-module structures on $V$ are defined similarly.

\subsubsection{Approximately unital modules}

For a general, not necessarily approximately unital, left $R_E(K)$-module structure $\rho:R_E(K)\to\End_E(V)$ we define the subspace of \lq{}$K$-finite vectors\rq{} as
\begin{equation}
V_{K}\;:=\;\varinjlim_C 1_C V\;\subseteq\;V.
\label{eq:kfinitev}
\end{equation}
This is the maximal approximately unital $R_E(K)$-submodule of $V$ and $V$ is an approximately unital $R_E(K)$-module if and only if $V_{K}=V$. As in Proposition 1.55 of \cite{book_knappvogan1995} we see that the functor $V\mapsto V_{K}$ is exact.

\begin{theorem}\label{thm:heckemodulesfork}
The category of rational left (resp.\ right) $K$-modules over $E$ is equivalent to the category of approximately unital left (resp.\ right) $R_E(K)$-modules.
\end{theorem}

\begin{proof}
Let us show that there are mutually inverse functors between the categories in question.

On the one hand, each right $\OO_E(K)$-comodule $V$ gives naturally rise to a rational $\OO_E(K)^*$-module structure on $V$, and by restriction we obtain an associative action of $R_E(K)$ on $V$. We need to show that this action is approximately unital.

We let $V\to V\otimes\OO_E(K)$ denote the comodule structure. Consider for any finite-dimensional subspace $W\subseteq V$. Then $\phi(W)\subseteq V\otimes\OO_E(K)$ is finite-dimensional and we know that coassociativity implies that there is a finite-dimensional cosubalgebra $C\subseteq\OO_E(K)$ and a subspace $V_C\subseteq V$ containing $W$ with the property that the restriction of $\phi$ to $V_C$ induces a $C$-comodule structure
$$
V_C\to V_C\otimes C.
$$
Then $V_C$ is a left $R_E(K)_C$-module and therefore $1_C W=W$. This show that $V$ is indeed an approximately unital $R_E(K)$-module.

Assume on the other hand, that $V$ is an approximately unital left $R_E(K)$-module. Let us construct a right $\OO_E(K)$-comodule structure on $V$. Let again $W\subseteq V$ denote an arbitrary finite-dimensional subspace. Then we find a finite-dimensional subcoalgebra $C\subseteq\OO_E(K)$, such that for the associated elementary idempotent $1_C$ we have $1_C W=W$. Consider the $R_E(K)_C$-module $V_C$ generated by $W$. $V_C$ naturally inherits a right $C$-comodule structure
$$
V_C\to V_C\otimes C.
$$
Then the image of $W$ under the composition 
$$
V_C\to V_C\otimes C\;\to\;V\otimes\OO_E(K)
$$
is independent of the choice of $C$, and therefore this construction determines a unique $\OO_E(K)$-comodule structure on $V$.

It is easy to check that the two given constructions are inverses to each other.

The proof for $R_E(K)$-right modules and left $\OO_E(K)$-comodules proceeds mutatis mutandis. This proves the theorem.
\end{proof}

The proof of Theorem \ref{thm:heckemodulesfork} also shows that $V_{K}$ agrees with the maximal rational $K$-submodule of $V$.

For later use we introduce the following notation. If $V$ is a finite-dimensional rational $K$-module, then there is a unique \lq{}minimal\rq{} approximate identity $1_V\in R_E(K)$ with the property that $1_VV=V$. It may be constructed explicitly as follows. For each irreducible rational $K$-submodule $Z$ over $E$ occuring in $V$ there is a unique cosimple cosubalgebra $C_Z\subseteq\OO_E(K)$ which gives rise to an approximate identity $1_{C_Z}\in R_E(K)$ with the property that $1_{C_Z}$ acts as identity on $Z$ and as zero on any irreducible $Z'$ which is not isomorphic to $Z$. Then $1_V$ is the finite sum of the $1_{C_Z}$, where $Z$ runs over representatives of the isomorphism classes of $E$-rational $K$-types occuring in $V$. If $V=0$ we adopt the convention that $1_V=0$.

\subsection{Hecke algebras for general pairs}

Given the results of the previous section, the definition of Hecke algebras for general pairs $(\liea,B)$ over $E$ is straightforward. As the universal enveloping algebra $U_E(\liea)$ over $E$ is a rational left and right $B$-module, it is naturally an approximately unital left and right $R_E(B)$-module. Therefore we may define the Hecke algebra as an appropriate quotient of the smash product
$$
R_E(B)\,\#\,U_E(\liea)\;:=\;R_E(B)\otimes_E U_E(\liea).
$$
The appropriate structure of associative algebra with approximate identities on the smash product for weak pairs over $\CC$ is discussed in detail in section 5 of chapter I of \cite{book_knappvogan1995}. The same construction remains valid in our situation if the roles of $C(K)$ and $\mathcal A$ are played by $\OO_E(K)$ and $U_E(\liea)$ respectively. As in Lemma 1.113 of section 6 of loc.\ cit.\ we define a two-sided ideal $I_{\lieb}$ inside the smash product and obtain the Hecke algebra for the pair $(\liea,B)$ as the factor ring
$$
R_E(\liea,B)\;:=\;\left(R_E(B)\,\#\,U_E(\liea)\right)/I_{\lieb}.
$$
As a vector space we have $R_E(\liea,B)=U_E(\liea)\otimes_{U_E(\lieb)}R_E(B)$.

Then $R_E(\liea,B)$ is again an associative algebra with approximate identity. The aproximate units are again numbered by finite-dimensional subcoalgebras $C$ of $R_E(B)$ and we denote them simply $1_C$, and set
$$
R_E(\liea,B)_C\;:=\;1_C\cdot R_E(\liea,B)\cdot 1_C.
$$
This is an $E$-algebra with unit $1_C$, and
$$
R_E(\liea,B)\;=\;\varinjlim_C R_E(\liea,B)_C
$$
in the category of associative $E$-algebras.

Then Theorem \ref{thm:heckemodulesfork} naturally generalizes and we obtain
\begin{theorem}\label{thm:heckemodulesforgk}
The category $\mathcal C(\liea,B)_E$ of $(\liea,B)$-modules over $E$ is equivalent to the category of approximately unital left $R_E(\liea,B)$-modules.
\end{theorem}

\begin{proof}
The proof proceeds mutatis mutandis as the classical proof over $\CC$, cf.\ proof of Theorem 1.117 of loc.\ cit.
\end{proof}

In the sequel we freely identify the categories of $(\liea,B)$-modules and approximately unital left $R_E(\liea,B)$-modules.

We remark that the antipodes of the Hopf algebras $U_E(\liea)$ and $\OO_E(B)$ induce a \lq{}transpose map\rq{} $(\cdot)^t$ on $R_E(\liea,B)$, which allows us to turn left modules into right modules and vice versa (cf.\ Proposition 1.122 of loc.\ cit.).

\subsection{Zuckerman and Bernstein functors}

In this section we define general induction and production functors for maps of pairs over $E$. The treatment follows chapter II of the monograph \cite{book_knappvogan1995} closely. All arguments there carry over to our situation without modification. Therefore we content us to give references to the statements in loc.\ cit., instead of reproducing essentially identical proofs.

The notion of a map of pairs $\iota:(\liea,B)\to (\lieb,H)$ over $E$ is defined as in the complex case. I.e.\ $\iota$ consists of a map of Lie algebras $\liea\to\lieb$ and a map of reductive groups $B\to H$ which are compatible in the obvious way. Then the natural left and right $B$-actions on $H$ induce left and right $R_E(B)$-actions on $R_E(H)$. Similarly we obtain actions of $U_E(\liea)$ on $U_E(\liec)$ and the compatibility of all these actions allows us to descend these actions to left and right actions of $R_E(\liea,B)$ on $R_E(\liec,H)$.

\subsubsection{Passage to $K$-finite vectors}

As in \eqref{eq:kfinitev} we let $(\cdot)_B$ denote the subspace of $B$-finite vectors in the approximate unital sense: If $V$ is a module for the asssociative algebra $R_E(B)$, then
$$
V_B\;=\;\{m\in M\mid \exists C: 1_C m = m\}.
$$
If $V$ is a non-unital $R_E(\liea,B)$-module, then $V_B$ is an approximately unital $R_E(\liea,B)$-module and this operation is an exact functor from the category of non-unital $R_E(\liea,B)$-modules to approximately unital $R_E(\liea,B)$-modules. We remark that with this notation
$$
V^\vee\;=\;\Hom_E(V,E)_B,
$$
for any $(\liea,B)$-module $V$ over $E$.

\subsubsection{The Bernstein functor}

Following section 1 of chapter II in \cite{book_knappvogan1995}, we define the functor
$$
P_{\iota,E}\;:=\;R_E(\liec,H)\otimes_{R_E(\liea,B)}(-)
$$
as a functor from $(\liea,B)$-modules to $(\liec,H)$-modules. If there is no risk of confusion, we occasionally drop the index $E$ from the notation.

Then $P_{\iota,E}$ is right exact and left adjoint to the \lq{}completed\rq{} forgetful functor $\mathcal F^\vee_{\iota,E}:\mathcal C(\liec,H)_E\to\mathcal C(\liea,B)_E$, defined as
$$
\mathcal F^\vee_{\iota,E}(-)\;:=\;\Hom_{R_E(\lieb,H)}(R_E(\lieb,H),-)_{B},
$$
(cf.\ Proposition 2.34 of loc.\ cit.). Since $\mathcal F^\vee_{\iota,E}$ is exact, $P_{\iota,E}$ sends projectives to projectives.

\subsubsection{The Zuckerman functor}

Dually to $P_{\iota,E}$ we define the left exact functor
$$
I_{\iota,E}\;:=\;\Hom_{R_E(\liea,B)}(R_E(\liec,H),-)_{H},
$$
which is a right adjoint to the standard forgetful functor $\mathcal F_{\iota,E}:\mathcal C(\liec,H)_E\to\mathcal C(\liea,B)_E$ along $\iota$ (cf.\ Proposition 2.21 of loc.\ cit.). The latter is again exact and therefore $I_{\iota,E}$ preserves injectives.

We remark that we have a monomorphism (Proposition 2.33 of loc.\ cit.),
$$
\mathcal F_{\iota,E}\to \mathcal F^\vee_{\iota,E}.
$$

\subsubsection{Standard resolutions}

The standard (Koszul) resolution in $\mathcal C(\liea,B)_E$ is given by
\begin{equation}
0\to X_{\dim{\liea/\lieb},E}\to\cdots \to X_{1,E}\to X_{0,E}\to E\to 0,
\label{eq:koszul}
\end{equation}
with
$$
X_{q,E}\;:=\;P_{(\lieb,B)\to(\liea,B)}\left(\bigwedge^q(\liea/\lieb)\right)\;=\;R_E(\liea,B)\otimes_{R_E(B)}\bigwedge^q(\liea/\lieb).
$$
The differentials and the augmentation map in \eqref{eq:koszul} are given as in (2.121c) and (2.121b) in section 7 of chapter II of loc.\ cit.

By applying the exact functors $(-)\otimes_E V$ and $\Hom_E(-,V)$ we deduce from \eqref{eq:koszul} standard projective and standard injective resolutions of $V$ in $\mathcal C(\liea,B)$.

For any map of pairs $\iota:(\liea,B)\to (\liec,H)$, we obtain the explicit complex
$$
C_q(V)_E\;:=\;R_E(\liec,H)\otimes_{R_E(\liec,\iota(B))} 
\bigwedge^q\left(\lieh/\iota(\lieb)\right)^*\otimes U_E(\liec)\otimes_{U_E(\liea)}V,
$$
whose homology computes $L^qP_{\iota,E}(V)$. If we are only interested in the $H$-action, the description of $C_q(V)_E$ simplifies to
$$
C_q(V)_E|_H\;=\;R_E(H)\otimes_{R_E(\iota(B))} 
\bigwedge^q\left(\lieh/\iota(\lieb)\right)^*\otimes U_E(\liec)\otimes_{U_E(\liea)}V.
$$
In particular $L^qP_{\iota,E}(V)=0$ whenever $q>\dim(\lieh/\iota(\lieb))$.

\subsection{Base change properties}

In this section we study the behavior of the preceeding constructions under base change. This plays a fundamental role in our applications.

\begin{proposition}\label{prop:heckebasechange}
For any map of fields $\sigma E\to F$ in characteristic $0$ and any pair $(\liea,B)$ over $E$ we have natural isomorphisms
\begin{equation}
R_E(B)\otimes_{E,\sigma}F\;\cong\;R_F(B\times_{\Spec E,\sigma}\Spec F)
\label{eq:compactheckeiso}
\end{equation}
and
\begin{equation}
R_E(\liea,B)\otimes_{E,\sigma}F\;\cong\;R_F(\liea\otimes_{E,\sigma}F,B\times_{\Spec E,\sigma}F)
\label{eq:heckeiso}
\end{equation}
of algebras with approximate identities.
\end{proposition}

\begin{proof}
It suffices to prove the existence of the isomorphism \eqref{eq:compactheckeiso}, as the existence of the isomorphism \eqref{eq:heckeiso} follows from the first by the invariance under base change of the construction of the Hecke algebras via smash products.

By our definition of the Hecke algebras for $B$ over $E$ and $F$, and since $\OO_E(B)\otimes_{E,\sigma}F\cong\OO_F(B\times_{\Spec E,\sigma}F)$ naturally, it suffices to show that every coideal $J$ in $\OO_F(B\times_{\Spec E,\sigma}F)$ contains the image under $\sigma$ of a coideal $J'$ in $\OO_E(B)$ of finite codimension, the other direction being obvious.

To see this, observe that in the decomposition \eqref{eq:ocodecomposition} over $E$, i.e.\
$$
\OO_E(B)\;=\;\bigoplus_\rho \OO_{\rho,E},
$$
each direct summand $\OO_{\rho,E}$ is finite-dimensional and decomposes, after application of the functor $-\otimes_{E,\sigma}F$ over $F$ into a finite sum of simple coalgebras $\OO_{\tau(\rho),F}$, and \eqref{eq:ocodecomposition} over $F$ then reads
$$
\OO_F(B)\;=\;\bigoplus_\rho\left(\bigoplus_{\tau(\rho)} \OO_{\tau(\rho),F}\right).
$$
Arguing as in the proof of Proposition \ref{prop:leftrightideals} we see with \eqref{eq:itaustructure} that
$$
J\;=\;\bigoplus_\rho\left(\bigoplus_{\tau(\rho)} (J\cap\OO_{\tau(\rho),F})\right),
$$
and for almost all indices $\rho$,
$$
(J\cap\OO_{\tau(\rho),F})\;=\;\OO_{\tau(\rho),F}.
$$
Therefore
$$
J'\;:=\;\sigma^{-1}(J\cap R_E(B)^\sigma)
$$
is a coideal in $\OO_E(B)$ satisfying
$$
(J'\cap\OO_{\rho,E})\;=\;\OO_{\rho,E}
$$
for all but finitely many indices $\rho$, and therefore is of finite codimension.
\end{proof}

\begin{corollary}\label{cor:forgetfulbasechange}
For any map of fields $\sigma: E\to F$ let $\iota_E:(\lieh,L)\to(\lieg,K)$ be a map of pairs over $E$. Consider the induced map
$$
\iota_F:(\lieh\otimes_{E,\sigma}F,L\times_{\Spec E,\sigma}\Spec F)\to(\lieg\otimes_{E,\sigma}F,K_{\Spec E,\sigma}F)
$$
of pairs over $F$. Then we have natural isomorphisms of functors
\begin{equation}
\mathcal F_{\iota_F}\circ (-\otimes_{E,\sigma}F)\;\cong\;
\mathcal (-\otimes_{E,\sigma}F)\circ \mathcal F_{\iota_E}
\label{eq:forgetfulbasechangeiso}
\end{equation}
and
\begin{equation}
\mathcal F_{\iota_F}^\vee\circ (-\otimes_{E,\sigma}F)\;\cong\;
\mathcal (-\otimes_{E,\sigma}F)\circ \mathcal F_{\iota_E}^\vee
\label{eq:forgetfulvbasechangeiso}
\end{equation}
\end{corollary}

The following Theorem is a generalization of Proposition 2.3 and Theorem 2.5 in \cite{januszewskipreprint}.
\begin{theorem}\label{thm:basechangeinduction}
With the notation from Corollary \ref{cor:forgetfulbasechange}, we have for any $(\lieh,L)$-module $V$ over $E$ and any degree $j$ natural isomorphisms
\begin{equation}
L_jP_{\iota_E}(V)\otimes_{E,\sigma}F\;\cong\;L_jP_{\iota_F}(V\otimes_{E,\sigma}F)
\label{eq:derivedbernsteiniso}
\end{equation}
and
\begin{equation}
R^jI_{\iota_E}(V)\otimes_{E,\sigma}F\;\cong\;R^jI_{\iota_F}(V\otimes_{E,\sigma}F)
\label{eq:derivedzuckermaniso}
\end{equation}
of $(\lieg\otimes_{E,\sigma}F,K\times_{\Spec E,\sigma}F)$-modules over $F$.
\end{theorem}

The isomorphisms in Theorem \ref{thm:basechangeinduction} also preserve the usual spectral sequences and K\"unneth isomorphisms.



\begin{proof}
In degree $j=0$ we know by \ref{prop:heckebasechange} and the definition of $P_{\iota_\bullet}$ and $I_{\iota_\bullet}$ that the statement is true (alternatively we may appeal to Corollary \ref{cor:forgetfulbasechange} and the adjointness relations). The statement for general degrees $j$ is a consequence of the general Homological Base Change Theorem (Theorem 2.1 in \cite{januszewskipreprint}). Alternatively one may argue that the standard Koszul resolutions are stable under base change and hence the standard complexes for computing the derived functors in question also commute with base change.
\end{proof}

\subsection{The algebraic Borel-Bott-Weil Theorem}

In this section we discuss the algebraic analogue of the Borel-Bott-Weil Theorem over non-algebraically closed fields $E$ of characteristic $0$. This will be useful in section \ref{sec:nonvanishingcriterion}.

Suppose $E$ fixed, and $K$ to be a reductive group over $E$ which is not necessarily connected. Assume a parabolic subalgebra $\lieq\subseteq\liek$ with Levi decomposition $\lieq=\liel+\lieu$ over $E$ is given, and assume furthermore that $\liel$ is the Lie algebra (of the connected component of the identity) of a reductive subgroup $L\subseteq K$ over $E$ which meets all connected components of $K$. Recall that we always assume connected components to be defined over $E$.

Choose an extension $F/E$ where all absolutely irreducible representations of $K$ are defined and where $K^0$ is split (the former doesn't imply the latter in general, $\Oo(2n)/\QQ$ is a counter example). We fix a Borel subalgebra $\lieb_F\subseteq\liek_F$ containing $\lieu$ and such that
$$
\lieb_F\;=\;(\lieb_F\cap\liel_F)+\lieu_F.
$$
Then $\lieb_F\cap\liel_F$ is a Borel in $\liel_F$ and we may fix a Cartan subalgebra $\liet_F\subseteq\lieb_F\cap\liel_F$, and obtain a corresponding Levi decomposition
$$
\lieb_F\cap\liel_F\;=\;\liet_F+\lien_F.
$$
Then if $Z$ is an irreducible $(\liel,L)$-module over $E$, which is not necessarily absolutely irreducible, the module $Z$ gives rise to a collection of $\lieb_F\cap\liel_F$-highest weights weights of $\liet_F$. Those are by definition the weights occuring in $H^0(\lien_F;Z_F)$.

We set $S_\lieq:=\dim\lieu$, and write
$$i:\quad(\lieq,L)\to(\liek,K)$$
for the canonical inclusions of pairs and
$$p:\quad(\lieq,L)\to(\liel,L)$$
for the canonical projections.

\begin{theorem}\label{thm:borelbottweil}
With the notation above, let $Z$ be an irreducible $(\liel,L)$-module over $E$ (not necessarily absolutely irreducible), whose $\lieb_F\cap\liel_F$-highest weights are $\lieb_F$-dominant. Then the $(\liek,K)$-modules
$$
L_{q}P_{i,E}(\mathcal F_{p}\left(Z\otimes \bigwedge^{S_\lieq}\lieu\right)^\vee),
$$
and
$$
R^{q}I_{i,E}(\mathcal F_{p}\left(Z\otimes \bigwedge^{S_\lieq}\lieu\right)),
$$
are non-zero if and only if $q=S_{\lieq}$. In that case they are both irreducible and characterized by the property that
$$
H_0(\lieu;L_{S_\lieq}P_{i,E}(\mathcal F_{p}\left(Z\otimes \bigwedge^{S_\lieq}\lieu\right)^\vee)\;\cong\;Z^\vee,
$$
and
$$
H^0(\lieu;R^{S_\lieq}I_{i,E}(\mathcal F_{p}\left(Z\otimes \bigwedge^{S_\lieq}\lieu\right))
\;\cong\;Z,
$$
respectively.
\end{theorem}

\begin{proof}
To prove the vanishing statement we may assume $E=F$, i.e.\ $K^0$ is split over $E$, as $L_{q}P_{i^-,E}$ and $R^{q}I_{i,E}$ commute with extension of scalars by Theorem \ref{thm:basechangeinduction}. In this case the proof proceeds identically as the proof of Proposition 4.173 in \cite{book_knappvogan1995}. This proof also shows the validity of Theorem when $K^0$ is split over $E$.

In the case where $K^0$ is not split over $E$, it remains to show that for irreducible but not absolutely irreducible $Z$ the resulting modules are irreducible again. We set
$$
X\;:=\;L_{q}P_{i,E}(\mathcal F_{p}\left(Z\otimes \bigwedge^{S_\lieq}\lieu\right)^\vee).
$$
Let us show that over $E$,
\begin{equation}
H^0(\lieu;X)\;\cong\;Z^\vee.
\label{eq:ucohomologyiso}
\end{equation}
The left hand side becomes isomorphic to $Z_F^\vee$ over $F$, since the statement is known for irreducibles over $F$ and all involved functors commute with direct sums. Now
$$
\Hom_{\liel,L}(Z^\vee,H_{0}(\lieu^-;X)\;\neq\;0
$$
by (4.170g) of loc.\ cit., which is valid over $E$, or alternatively by Proposition 1.1 in \cite{januszewskipreprint}, since this Hom-space is non-zero over $F$. This then implies the existence of the isomorphism \eqref{eq:ucohomologyiso} over $E$.

Now for any non-zero $(\liek,K)$-factor module $Y$ of $X$ over $E$ we obtain an epimorphism
$$
H_0(\lieu;X)\;\to\;H_0(\lieu;Y).
$$
Since $Y$ is non-zero the right hand side is non-zero, because it is clearly non-zero over $F$. Therefore the irreducibility of the left hand side implies $Y=X$. Hence $X$ is irreducible.

The case of the Zuckerman functor is treated mutatis mutandis.
\end{proof}

\section{Hard Duality}

In this section we show that hard duality holds over any base field $E$ of characteristic $0$ and is natural in $E$. The technical hypothesis which makes hard duality work in the classical situation is the compactness of the groups $K$ and $L$ occuring in the pairs in question, as this leads, via the the existence and uniqueness of Haar measures, to a natural duality theory between $C(K)$ and $R(K)$. This duality may then be extended to the complexes computing the derived functors of $P_\iota$ and $I_\iota$, which ultimately leads to the duality statement.

In our situation the compactness property is replaced by the reductiveness of the linear algebraic groups $K$ and $L$. Here the duality manifests itself in a natural duality between $\OO_E(K)$ and $R_E(K)$, provided by the abstract algebraic notion of integral on Hopf algebras. The notion of left and right integral may be defined for any Hopf algebra. However, since $K$ is reductive, $\OO_E(K)$ is cosemisimple and this implies that left and right integrals do exist and in fact span the same one-dimensional space inside $\OO_E(K)^*$. We will first recall these classical facts, then state the Hard Duality Theorem, and show how the classical proof given in Chapter III of \cite{book_knappvogan1995} carries over to our situation without modification.

\subsection{Integrals on Hopf algebras}

As in section \ref{sec:reductivehecke} our standard reference for the material of this section is \cite{book_sweedler1969}. All results we need are to be found in Chapters V, XIV, and XV of loc.\ cit.

We let $K$ denote a reductive group over a field $E$ of characteristic $0$ and write $1_\OO$ for the identity in $\OO_E(K)$. For $\xi_1,\xi_2\in\OO_E(K)^*$ we write $\xi_1\cdot\xi_2$ for the {\em algebra product} of the two linear forms in the algebra $\OO_E(K)^*$. This algebra structure is inherited from the coalgebra structure on $\OO_E(K)$.

An element $\lambda\in\OO_E(K)^*$ is called a {\em left integral} (on $K$) if for any $\xi\in\OO_E(K)^*$ we have
$$
\xi\cdot\lambda\;=\;\xi(1_\OO)\lambda.
$$
Similarly $\lambda$ is called a {\em right integral} if for any $\xi\in\OO_E(K)^*$
$$
\lambda\cdot\xi\;=\;\xi(1_\OO)\lambda.
$$
We write $\Int_E(K)$ for the space of left integrals on $K$ over $E$. As $K$ is reductive, $\Int_E(K)$ is also the space of right integrals on $K$ and is a one-dimensional sub-$E$-vectorspace of $R_E(K)$.

\begin{proposition}\label{prop:integraliso}
Let $K$ be a reductive group over $E$. Then we have a canonical isomorphism
\begin{equation}
R_E(K)\;\cong\;\Int_E(K)\otimes_E \OO_E(K)
\label{eq:integraliso}
\end{equation}
of $\OO_E(K)$-Hopf modules. This isomorphism is natural in $E$, i.e.\ it commutes with base change.
\end{proposition}

\begin{proof}
The existence of the canonical isomorphism \eqref{eq:integraliso} is a standard fact in the theory of Hopf algebras (cf.\ chapter V of loc.\ cit.). The construction of the map
$$
\Int_E(K)\otimes_E \OO_E(K)\;\to\;R_E(K)
$$
on p.\ 97 of loc.\ cit.\ shows that \eqref{eq:integraliso} commutes with base change.
\end{proof}

\subsection{Hard Duality over general bases}\label{sec:hardduality}

We consider any map of pairs $\iota:(\lieh,L)\to (\lieg,K)$ over $E$ such that $\iota(\lieh)=\lieg$. Let $\liec$ be an $L$-invariant complement of $\iota^{-1}(\liek)$ inside $\liel$ and write $m$ for its dimension. Since $L$ acts on $\liec$, it also acts on $\bigwedge^m\liec$. We extend this action to $\lieh$ and thus to $R_E(\lieh,L)$ as in the classical case (cf.\ page 183 of \cite{book_knappvogan1995}).

\begin{theorem}[Hard Duality]\label{thm:hardduality}
For any map of pairs $\iota:(\lieh,L)\to (\lieg,K)$ over $E$ such that $\iota(\lieh)=\lieg$, and any choice of integral $0\neq dk\in\Int_E(K)$, we have for any $(\lieh,L)$-module $V$ isomorphisms
\begin{eqnarray*}
\mathcal D_{\iota,E}^{dk}:\quad L_jP_{\iota,E}(V\otimes_E\left(\bigwedge^m\liec\right)^*) &\cong& R^{m-j}I_{\iota,E}(V),\\
L_jP_{\iota,E}(V\otimes_E\bigwedge^m\liec)^\vee &\cong& L_{m-j}P_{\iota,E}(V^\vee),\\
R^jI_{\iota,E}(V\otimes_E\bigwedge^m\liec)^\vee &\cong& R^{m-j}I_{\iota,E}(V^\vee),
\end{eqnarray*}
where in all cases $0\leq j\leq m$. Furthermore all these isomorphisms are natural in $V$ and $E$, and depend linearly on $dk$, and for any map $\sigma:E\to F$ of fields we have a commutative square
$$
\begin{CD}
L_jP_{\iota,E}(V\otimes_E\left(\bigwedge^m\liec\right)^*) @>\mathcal D_{\iota,E}^{dk}>> R^{m-j}I_{\iota,E}(V)\\
@V{\sigma}VV @VV{\sigma}V\\
L_jP_{\iota^\sigma,F}(V^\sigma\otimes_F\left(\bigwedge^m\liec^\sigma\right)^*) @>>\mathcal D_{\iota^\sigma,F}^{dk^\sigma}> R^{m-j}I_{\iota^\sigma,F}(V^\sigma)
\end{CD}
$$
where
$$
V^\sigma\;=\;V\otimes_{E,\sigma}F,
$$
and
$$
\iota^\sigma:\quad(\lieh^\sigma,L^\sigma)\to (\lieg^\sigma,K^\sigma)
$$
is the induced map of pairs over $F$ under the functor $-\otimes_{E,\sigma}F$.
\end{theorem}

\begin{proof}
We remark that it suffices to prove the first isomorphism, as the other two are easy consequences of the first combined with Easy Duality.

We first prove the Theorem for the case $\lieg=\liek$. As in Lemma 3.32 of loc.\ cit.\ we consider for any trivial rational $K$-module $U$ the map
\begin{equation}
\OO_E(K)dk\otimes_E U\;\to\;\Hom_E(R_E(K),U)_K,
\label{eq:lemma332iso}
\end{equation}
given by
$$
fdk\otimes u\;\mapsto\;\lambda'_{f\otimes u}
$$
with
$$
\lambda'_{f\otimes u}:\quad T\mapsto T(f^t)u.
$$
By Proposition \ref{prop:integraliso} it is readily seen that \eqref{eq:lemma332iso} is an isomorphism of rational $K$-modules, which is natural in $E$ and commutes with $\sigma$ as claimed.

We let $\iota_K:(\iota^{-1}(\liek),L)\to(\liek,K)$ be the map of pairs induced by $\iota$ on the \lq{}compact pairs\rq{}. Let $V$ be any $(\liek,K)$-module. Analogously to (3.31a) and (3.31b) in loc.\ cit., the isomorphism \eqref{eq:lemma332iso} gives rise to an isomorphism
$$
\OO_E(K)dk\otimes\bigwedge^j\liec\otimes\mathcal F_{\iota_K}(V)\otimes\left(\bigwedge^m\liec\right)^*\to
\Hom_E(R_E(K),\Hom_E\bigwedge^{m-j}\liec,\mathcal F_{\iota_K}(V)))_K,
$$
of rational $K$-modules, where all tensor products are taken over $E$. This map is explicitly given by
$$
fdk\otimes\xi\otimes v\otimes\varepsilon\;\mapsto\;\lambda_{f\otimes\xi\otimes v\otimes\varepsilon}
$$
with
$$
\lambda_{f\otimes\xi\otimes v\otimes\varepsilon}(T)(\gamma)\;:=\;\varepsilon(\xi\wedge\gamma)T(f^t)v.
$$
This map descends to $L$-invariants, and hence induces a $K$-isomorphism
\begin{eqnarray*}
\lambda_j:&&\OO_E(K)dk\otimes_{R_E(L)}\left(\bigwedge^j\liec\otimes\mathcal F_{\iota_K,E}(V)\otimes\left(\bigwedge^m\liec\right)^*\right)\\
&&\to\Hom_{L}(R_E(K),\Hom_E\bigwedge^{m-j}\liec,\mathcal F_{\iota_K,E}(V)))_K,
\end{eqnarray*}
We let $\partial$ and $d$ denote the differentials of the complexes computing $L_jP_{\iota_K}$ and $R^jI_{\iota_K}$ as in (3.27) and (3.28) of loc.\ cit., respectively. Then we have the identity
$$
\lambda_{j-1}\circ\partial\;=\;(-1)^jd\circ\lambda_j,
$$
which shares the same proof, line by line, as the identity (3.35) in loc.\ cit. Therefore the isomorphisms $\lambda_j$ descend to homology and cohomology respectively. This proves the existence of a natural isomorphism
$$
L_jP_{\iota_K,E}(V\otimes_E\left(\bigwedge^m\liec\right)^*) \;\cong\; R^{m-j}I_{\iota_K,E}(V),
$$
which evidently commutes with base change. This completes the proof of hard duality as an isomorphism of $K$-modules.

Now if $E$ is a subfield of $\CC$, which is the case of interest for our number theoretic applications, we may reduce general case to the classical case over $\CC$ as follows. We first observe that Hard Duality over $\CC$ for $\iota:(\lieh,L)\to(\lieg,K)$ is compatible with Hard Duality for the restriction $\iota_K:(\iota^{-1}(\liek),L)\to(\liek,K)$ on compact pairs (cf.\ Proposition 2.69 and section 5 of chapter VI of loc.\ cit., in particular (6.43) and (6.44)). Therefore, to check that the $K$-equivariant isomorphism we constructed over $E$ is $\lieg$-equivariant, we may extend scalars to $\CC$, where we know it to be $\lieg_\CC$-equivariant by Theorem 3.5 of loc.\ cit.. This then already implies the claim over $E$.

It seems worth noting that the proof of the general case for general base fields $E$ may proceed as in \cite{book_knappvogan1995} as well. The difficulty lies in the description of the $\lieg$ action, which, on the complexes at hand, may only be described up to homotopy. This renders the final argument lengthy and technical, and we restrict ourselves to discuss the essential steps, again following loc.\ cit.

As in section 8 of chapter III of loc.\ cit.\ the proof readily reduces to two extreme cases:
\begin{itemize}
\item[(i)] The map $\iota:L\to K$ is onto.
\item[(ii)] The map $\iota:(\lieg,L)\to (\lieg,K)$ is an inclusion of pairs.

The proof of case (i) proceeds as follows. First Knapp and Vogan relate $P_\iota$ and $I_\iota$ to suitable $(\lies,M)$-homology and $(\lies,M)$-cohomology respectively. More concretely, let $\lies$ denote the kernel of $\iota:\lieh\to\lieg$ and let $M$ denote the kernel of $\iota:L\to K$. Then there are isomorphisms
\begin{eqnarray*}
\mathcal F_{(0,1)\to(\lieg,K),E}L_jP_{\iota,E}(V)&\cong& H_j(\lies,M;\mathcal F_{(\lies,M)\to(\lieh,L),E}(V)),\\
\mathcal F_{(0,1)\to(\lieg,K),E}R^jI_{\iota,E}(V)&\cong& H^j(\lies,M;\mathcal F_{(\lies,M)\to(\lieh,L),E}(V)),
\end{eqnarray*}
\end{itemize}
of $E$-vector spaces, and we may recover the $(\lieg,K)$-actions on the right hand sides via an explicit description on the standard complexes for $(\lies,M)$-(co)homology as in (3.43a) and (3.43b) of loc.\ cit.. Indeed, the first part of the proof of Proposition 3.41 of loc.\ cit.\ is completely formal and categorical, and, with what we have already proved, carries over to our setting without change. This shows that we indeed have isomorphisms of $E$-vector spaces. The verification of the $(\lieg,K)$-equivariance follows from a concise calculation, which remains valid in our case as well. As for $(\lies,M)$-(co)homology Hard Duality is a straightforward explicit calculation (cf.\ pages 203--205 of loc.\ cit.), this then proves case (i).

The proof of case (ii), which is in a sense the most interesting one, is based on similar arguments. The categorification of the $\lieg$-action on an $E$-vector space $M$ as an $E$-linear $\lieg$-equivariant map
$$
\mu:\quad\lieg\otimes_E M\to M
$$
or
$$
\mu:\quad M\to\Hom_E(\lieg,M)
$$
is straightforward (cf.\ page 216 of loc.\ cit.). Proposition 3.77 of loc.\ cit.\ (due to Enright-Wallach) describes the $\lieg$-action on the derived Bernstein and Zuckerman functors, and its proof readily extends to general base fields $E$. Duflo-Vergne's description of the vertical maps in the crucial diagrams (3.79a) and (3.79b) of loc.\ cit., again remains valid with the same proof for any base field $E$. Given this, the verification of the $\lieg$-equivariance is a formal computation, excercised in section 7 of chapter III of loc.\ cit., and valid for general $E$ again. This then completes the proof of case (ii).
\end{proof}

\section{Cohomologically induced functionals}

We give a functorial description of Sun's cohomologically induced functionals from \cite{sunpreprint2} based on base change morphisms and Hard Duality, valid over any field of characteristic $0$.

\subsection{Base Change Morphisms}

We consider a commutative square of arbitrary pairs
\begin{equation}
\begin{CD}
(\lieg,K)@<j_K<<(\liel,M)\\
@Ai_{\lieq}AA @Ai_{\liep}AA\\
(\lieq,L)@<j_L<<(\liep,N)\\
\end{CD}
\label{eq:basechangesquare}
\end{equation}
over any field $E$ of characteristic $0$. For any $(\lieq,L)$-module $X$ over $E$, the image of the identity under the map
$$
\Hom_{\liep,N}(\mathcal F_{j_L}(X),\mathcal F_{j_L}(X))\to
\Hom_{\liep,N}(\mathcal F_{j_L}\mathcal F_{i_\lieg}I_{i_\lieq}(X)
,\mathcal F_{j_L}(X))
$$
induced by the counit $\mathcal F_{i_\lieq}I_{i_\lieq}\to 1$ gives us a natural $(\liel,M)$-morphism, the {\em base change map}
\begin{equation}
\alpha:\;\;\;\mathcal F_{j_K}I_{i_\lieq}(X)\to I_{i_\liep}\mathcal F_{j_L}(X),
\label{eq:basechangealpha}
\end{equation}
by the commutativity of the above diagram. Dually we obtain a natural {\em base change map}
\begin{equation}
\beta:\;\;\;P_{i_\liep}\mathcal F^\vee_{j_L}(X)\to\mathcal F^\vee_{j_K}P_{i_{\lieq}}(X).
\label{eq:basechangebeta}
\end{equation}
We remark that these maps are natural in the same way the involved functors behave naturally in commutative squares and in $E$ (we suppressed $E$ as subscript of the functors).

We also introduce for maps of fields $\sigma:E\to F$ the notational convention
$$
\lieg^\sigma\;:=\;\lieg\otimes_{E,\sigma} F,\quad K^\sigma\;:=\;K\times_{\Spec E,\sigma}\Spec F,
$$
and likewise for other pairs.

\begin{lemma}\label{lem:rationalbasechange}
Then for any map of fields $\sigma:E\to F$, we have commutative squares
\begin{equation}
\begin{CD}
\mathcal F_{j_K}I_{i_\lieq}(X)@>\alpha>> I_{i_\liep}\mathcal F_{j_L}(X)\\
@V{\sigma}VV @VV{\sigma}V\\
\mathcal F_{j_{K^\sigma}}I_{i_{\lieq^\sigma}}(X^\sigma)@>>\alpha^\sigma> I_{i_{\liep^\sigma}}\mathcal F_{j_{L^\sigma}}(X^\sigma)
\end{CD}
\label{eq:alphasigmabasechange}
\end{equation}
and
\begin{equation}
\begin{CD}
P_{i_\liep}\mathcal F^\vee_{j_L}(X)@>\beta>>\mathcal F^\vee_{j_K}P_{i_{\lieq}}(X)\\
@V{\sigma}VV @VV{\sigma}V\\
P_{i_{\liep^\sigma}}\mathcal F^\vee_{j_{L^\sigma}}(X^\sigma)@>>\beta^\sigma>\mathcal F^\vee_{j_{K^\sigma}}P_{i_{\lieq^\sigma}}(X^\sigma)\\
\end{CD}
\label{eq:betasigmabasechange}
\end{equation}
Furthermore $\alpha^\sigma$ and $\beta^\sigma$ agree with the base change morphisms associated to the image of the diagram \eqref{eq:basechangesquare} under the functor $-\otimes_{E,\sigma}F$.
\end{lemma}

\begin{proof}
Applying $\sigma$ to the commutative square of pairs \eqref{eq:basechangesquare} gives us another commutative square
\begin{equation}
\begin{CD}
(\lieg^\sigma,K^\sigma)@<j_{K^\sigma}<<(\liel^\sigma,M^\sigma)\\
@A{i_{\lieq^\sigma}}AA @A{i_{\liep^\sigma}}AA\\
(\lieq^\sigma,L^\sigma)@<j_{L^\sigma}<<(\liep^\sigma,N^\sigma)
\end{CD}
\label{eq:basechangesquare2}
\end{equation}

We claim that the base change maps $\alpha'$ and $\beta'$ associated to \eqref{eq:basechangesquare2} agree with $\alpha^\sigma$ and $\beta^\sigma$, the images of \eqref{eq:basechangealpha} and \eqref{eq:basechangebeta} under $\sigma$ respectively.

By Corollary \ref{cor:forgetfulbasechange} and Theorem \ref{thm:basechangeinduction} we know that we have a natural isomorphism
$$
\mathcal F_{j_K}I_{i_\lieq}(X)^\sigma\;=\;
\mathcal F_{j_{K^\sigma}}I_{i_{\lieq^\sigma}}(X^\sigma),
$$
that we write as identity by abuse of notation, and likewise for all other objects occuring in \eqref{eq:alphasigmabasechange}. Under this correspondence we have
$$
i_{\lieq}^\sigma\;=\;i_{\lieq^\sigma}
$$
on the level of maps, and again similarly for all other maps in \eqref{eq:basechangesquare2}.

The same argument shows that the natural map
$$
\Hom_{\liep,N}(\mathcal F_{j_L}(X),\mathcal F_{j_L}(X))\to
\Hom_{\liep,N}(\mathcal F_{j_L}\mathcal F_{i_\lieg}I_{i_\lieq}(X)
,\mathcal F_{j_L}(X))
$$
commutes with $\sigma$, as does the counit $\mathcal F_{i_\lieq}I_{i_\lieq}\to 1$. Therefore the claim follows in the first case, and the second case is proved similarly.
\end{proof}

\subsection{Cohomologically induced morphisms}\label{sec:inducedmaps}

With a view towards our applications, assume now that $E^+$ is a field embedded into $\RR$. Consider an inclusion of reductive pairs $i_L:(\lieh,L)\to(\lieg,K)$ over $E^+$. In particular $\liel$ is an $E^+$-form of a $\theta$-stable real subalgebra $\liel_\RR\subseteq\lieg_\RR$. We write $\liel$ and $\liek$ for the Lie algebras of $L$ and $K$ respectively. We assume further given a quadratic extension $E/E^+$, whose completion at the extension of the archimedean places extendingt the real place $E^+\to\RR$ is the field $\CC$ of complex numbers. Assume furthermore we are given $\theta$-stable parabolic subpairs $(\lieq,C)$ of $(\lieg,K)$ and $(\liep,D)$, all defined over $E$, with associated Levi pairs $(\liec,C)$ and $(\lied,D)$ (automatically defined over $E^+$) and nilpotent radicals $\lieu$ and $\liev$ (defined over $E$).

In the sequel we consider all these data over the field $E$ without further mention. We call the non-trivial automorphism of the extension $E/E^+$ simply complex conjugation. It is the restriction of the honest complex conjugation on $\CC$ to $E$ via our fixed embeddings.

We consider the commutative diagram of pairs
\begin{equation}
\begin{CD}
(\lieg,K)@<j_{K}<<(\lieh,L)@>1>>(\lieh,L)\\
@Ai_{\lieg}AA @AAi_{\lieg\cap \liel}A @AAi_{\liel}A\\
(\lieg,C)@<j_{C\cap L}<<(\lieh,C\cap D)@>j_{D}>>(\lieh,D)\\
@Ai_{\lieq}AA @Ai_{\lieq\cap\liep}AA @AAi_{\liep}A\\
(\lieq,C)@<k_{C}<<(\lieq\cap\liep,C\cap D)@>k_{D}>>(\liep,D)\\
\end{CD}
\label{diag:gtol}
\end{equation}
In this diagram all maps are inclusions that we denote as indicated. We number the small commutative squares from left to right and top to bottom and denote the corresponding base change maps by $\alpha_1$, $\alpha_2$, etc.

By our assumptions complex conjugation sends $\lieq$ and $\liep$ to their corresponding opposites, and we denote the resulting inclusions by $i_{\overline{\lieq}}$ and $i_{\overline{\liep}}$ respectively. We define the degrees
$$
S_\lieq\;:=\;\dim\liek/\liek\cap\lieq,
$$
and similarly
$$
S_\liep\;:=\;\dim\liel/\liel\cap\liep,
$$
$$
T_{\lieq\cap\liep}\;:=\;\dim\liel/\liec\cap\lied.
$$
We assume that the three conditions
\begin{equation}
\lieg\;\;=\;\;\lieh+\lieq,
\label{cond:sum}
\end{equation}
\begin{equation}
\lieh\cap\lieq\;\;=\;\;\liep\cap\lieq,
\label{cond:int}
\end{equation}
\begin{equation}
\liel\cap\lieq\cap\liep\;\;=\;\;\liec\cap\lied,
\label{cond:num}
\end{equation}
are satisfied. Conditions \eqref{cond:sum} and \eqref{cond:int} are essential for the construction, and under those two conditions, condition \eqref{cond:num} becomes equivalent to
$$
T_{\lieq\cap\liep}\;=\;S_\lieq\;+\;S_\liep,
$$
which means that the defect in the degrees in our construction vanishes. This is a condition one may drop but for the applications we have in mind it is essential.

We define the dualizing modules
$$
D_{\liek}\;:=\;
\bigwedge^{2S_{\lieq}}(\liek/\liec),
$$
mutatis mutandis
$$
D_{\liel}\;:=\;
\bigwedge^{2S_{\liep}}(\liel/\lied),
$$
and finally
$$
D_{\lieq\cap\liep}\;:=\;
\bigwedge^{T_{\lieq\cap\liep}}(\liel/\liec\cap\lied),
$$
with the respective actions of $(\lieg,C)$, $(\lieh,D)$, and $(\lieh, C\cap D)$ as in section \ref{sec:hardduality}. We remark that $D_{\liek}$ and $(D_{\liel})$ are trivial modules, and we assume that the same is true for $D_{\lieq\cap\liep}$. Then we may consider all these dualizing modules as trivial $(\lieg,K)$ and trivial $(\lieh,L)$-modules respectively. This allows us also to fix an isomorphism
\begin{equation}
\pi^*:\quad D_{\lieq\cap\liep}^*\;\to\;D_{\liel}^*
\label{eq:dualizingiso}
\end{equation}
of $(\lieh,L)$-modules over $E^+$.

We fix an $E$-rational $(\lieq,C)$-module $X$ and an $E$-rational $(\liep,D)$-module $Y$, and both with trivial actions of the nilpotent radicals, and assumed admissible for the actions of the Levi factors.

For any $(\lieq\cap\liep,C\cap D)$-equivariant map
$$
\eta_{\lieq\cap \liep}:\;\;\;\mathcal F_{k_{C}}(X_\lieq)\;\to\;\mathcal F^\vee_{k_{D}}(Y_\liep),
$$
we are going to construct for each $q\in\ZZ$ a natural $(\lieh,L)$-equivariant map
$$
\eta_{\lieh,q}:\;
\mathcal F_{j_{K}}L_{S_\lieq+q}P_{i_\lieg\circ i_{\overline{\lieq}}}(X)\otimes D_\liek^*
$$
$$
\;\to\;
L_{S_\liep+q}P_{i_\liel\circ i_{\overline{\liep}}}(Y)\otimes D_{\liel}^*.
$$
Naturality here is understood with respect to $X,Y$, $\eta_{\lieq\cap\liep}$, $\pi^*$, and $E$, which means in particular that our construction is functorial in the data given by diagram \ref{diag:gtol}. We will exploit this to study the effect on the bottom layer, which is crucial for our applications.

We fix an integral $dk\in\Int_E(K)$. Then Hard Duality provides us with an isomorphism
$$
\mathcal D_\lieg:=\mathcal D_{i_\lieg}^{dk}:\;\;\;
(L_{S_\lieq+q}P_{i_\lieg})P_{i_{\overline{\lieq}}}(X)\otimes D_{\liek}^*\to
(R^{S_\lieq-q}I_{i_\lieg})P_{i_{\overline{\lieq}}}(X).
$$
We have the natural map
$$
\alpha_1:\;\;\;
\mathcal F_{j_K}(R^{S_\lieq-q}I_{i_\lieg})P_{i_{\overline{\lieq}}}(X)\to
R^{S_\lieq-q}(I_{i_{\lieg\cap\liel}}\mathcal F_{j_{C\cap L}})P_{i_{\overline{\lieq}}}(X),
$$
induced by base change, as $I_{i_\lieg}$ sends injectives to injectives and the functor $\mathcal F_{j_\lieg}$ is exact. Since right derived functors are universal $\delta$-functors in the sense of \cite{grothendieck1957}, we get a natural map
\begin{equation}
\delta:\;\;\;
R^{S_\lieq-q}(I_{i_{\lieg\cap\liel}}\mathcal F_{j_{C\cap L}})P_{i_{\overline{\lieq}}}(X)\to
(R^{S_\lieq-q}I_{i_{\lieg\cap\liel}})\mathcal F_{j_{C\cap L}}P_{i_{\overline{\lieq}}}(X).
\label{eq:deltamap}
\end{equation}
We invoke Hard Duality again (fixing an integral $dl\in\Int_E(L)$) and obtain an isomorphism
$$
\mathcal D_{\lieg\cap\liel}^{-1}:\;\;\;
(R^{S_\lieq-q}I_{i_{\lieg\cap\liel}})\mathcal F_{j_{C\cap L}}P_{i_{\overline{\lieq}}}(X)
$$
$$
\to\;\;\;
(L_{S_\liep+q}P_{i_{\lieg\cap\liel}})
\mathcal F_{j_{C\cap L}}P_{i_{\overline{\lieq}}}(X)\otimes D_{\lieq\cap\liep}^*
$$
by condition \eqref{cond:num}. The third commutative square provides us by \eqref{cond:sum} and \eqref{cond:int} with a base change map
$$
\beta_3^{-1}\otimes 1_{D_{\lieq\cap\liep}^*}:\;\;\;
\mathcal F_{j_{C\cap L}}P_{i_{\overline{\lieq}}}(X)\otimes D_{\lieq\cap\liep}^*\;\to\;
P_{i_{\lieq\cap\liep}}\mathcal F_{k_{C\cap L}}(X)\otimes D_{\lieq\cap\liep}^*.
$$
Now we are in a position to apply $\eta_{\lieq\cap\liep}$ inside the argument and $\pi^*$ on the second outer tensor factor, which yields the map
$$
P_{i_{\lieq\cap\liep}}\mathcal F_{k_{C\cap L}}(\eta_{\lieq\cap\liep})\otimes\pi^*:\quad
P_{i_{\lieq\cap\liep}}\mathcal F_{k_{C\cap L}}(X)\otimes D_{\lieq\cap\liep}^*\;\to\;
P_{i_{\lieq\cap\liep}}\mathcal F_{k_{D}}^\vee(Y)\otimes D_{\liel}^*.
$$
From there we consider the base change map
$$
\beta_4\otimes 1_{D_{\liel}^*}:\;\;\;
P_{i_{\lieq\cap\liep}}\mathcal F_{k_{D}}^\vee(Y)\otimes D_{\liel}^*\;\to\;
\mathcal F_{j_{D}}^\vee P_{i_{\overline{\liep}}}(Y)\otimes D_{\liel}^*.
$$
Another base change gives, together with the same universal $\delta$-functor argument as above,
$$
\beta_2\otimes 1_{D_{\liel}^*}:\;\;\;
(L_{S_\liep+q}P_{i_{\lieg\cap\liel}})
\mathcal F_{j_{D}}^\vee P_{i_{\overline{\liep}}}(Y)\otimes D_{\liel}^*
\;\;\to\;\;
L_{S_\liep+q}P_{i_{\liel}\circ i_{\overline{\liep}}}(Y)\otimes D_{\liel}^*.
$$
Finally we define $\eta_{\lieh,q}$ as the composition of these maps:
$$
\eta_{\lieq,q}\;:=\;
\left((\beta_2\circ\beta_4\circ P_{i_{\lieq\cap\liep}}\mathcal F_{k_{C\cap L}}(\eta_{\lieq\cap\liep})\circ\beta_3^{-1})\otimes\pi^*\right)\circ
\mathcal D_{\lieg\cap\liel}^{-1}\circ\delta\circ\alpha_1\circ\mathcal D_{\lieg}.
$$
Then $\eta_{\lieh,q}$ is $(\lieh,L)$-equivariant by construction, since all individual maps are $(\lieh,L)$-equivariant.

\subsection{Functoriality properties}

In this section we study consequences of the functoriality properties of cohomologically induced maps.

\subsubsection{Rationality properties}

The construction of the map $\eta_{\lieh,q}$ commutes with base change in the following sense.

\begin{theorem}\label{thm:etabasechange}
Assume we are given the data as in section \ref{sec:inducedmaps} necessary for the construction of $\eta_{\lieh,q}$. Then for any map of fields $E\to F$ in characteristic $0$, we have a commutative square
$$
\begin{CD}
\mathcal F_{j_{K}}L_{S_\lieq+q}P_{i_\lieg\circ i_{\overline{\lieq}}}(X)\otimes D_\liek^*
@>\eta_{\lieh,q}>>
L_{S_\liep+q}P_{i_\liel\circ i_{\overline{\liep}}}(Y)\otimes D_{\liel}^*\\
@V{\sigma}VV @V{\sigma}VV\\
\mathcal F_{j_{K^\sigma}}L_{S_\lieq+q}P_{i_{\lieg^\sigma}\circ i_{\overline{\lieq}^\sigma}}(X)\otimes D_{\liek^\sigma}^*
@>>\eta_{\lieh^\sigma,q}>
L_{S_\liep+q}P_{i_{\liel^\sigma}\circ i_{\overline{\liep}^\sigma}}(Y)\otimes D_{\liel^\sigma}^*
$$
\end{CD}
$$
where $\eta_{\lieh^\sigma,q}$ is constructed with respect to the maps
$$
\eta_{\lieq^\sigma\cap \liep^\sigma}:\;\;\;\mathcal F_{k_{C^\sigma\cap K^\sigma}}(X_{\lieq^\sigma}^\sigma)\;\to\;\mathcal F^\vee_{k_{D^\sigma\cap L^\sigma}}(Y_{\liep^\sigma}^\sigma),
$$
$$
\left(\pi^*\right)^\sigma:\quad D_{\lieq^\sigma\cap\liep^\sigma}^*\to D_{\liel^\sigma}^*,
$$
and the integrals $dk^\sigma\in\Int_F(K^\sigma)$ and $dl^\sigma\in\Int_F(L^\sigma)$.
\end{theorem}

\begin{proof}
We may ignore the fact that the inclusions $E\to\CC$ may get mapped under $\sigma$ to a field which does not map into $\CC$, and that the image of $E^+$ under the functor $-\otimes_{E,\sigma}F$ is not well defined in general, since the field $E^+$ plays no role in the construction and is only used as a reference for the notion of $\theta$-stable parabolic pairs.

Given this, observe that applying the functor
$$
(-)^\sigma\;:=\;(-)\otimes_{E,\sigma}F
$$
to the diagram \eqref{diag:gtol} yields another diagram which satisfies the conditions \eqref{cond:sum}, \eqref{cond:int} and \eqref{cond:num}, and leaves the degrees $S_\lieq$, $S_\liep$, and $T_{\lieq\cap\liep}$ invariant, and the images of the dualizing modules under $(-)^\sigma$ are again trivial.

Therefore $\eta_{\lieh^\sigma,q}$ is well defined. It remains to show that it agrees with the image of $\eta_{\lieh,q}$ under $\sigma$ as claimed.

By Corollary \ref{cor:forgetfulbasechange} the forgetful functors commute with $\sigma$, by Theorem \ref{thm:basechangeinduction} the derived induction and production functors commute with $\sigma$, by Theorem \ref{thm:hardduality} Hard Duality commutes with $\sigma$, and by Lemma \ref{lem:rationalbasechange} the base change maps $\alpha_j$, $\beta_j$ commute with $\sigma$ as well. Therefore it only remains to remark that the map $\delta$ in \eqref{eq:deltamap} commutes with $\sigma$, which is obviously the case.

As a composition of base change maps, Hard Duality isomorphisms, and $\delta$, the cohomologically induced map $\eta_{\lieh,q}$ maps under $(\cdot)^\sigma$ to $\eta_{\lieh^\sigma,q}$ as claimed.
\end{proof}

\subsubsection{Restriction to the bottom layer}

One of the main features of $\eta_{\lieh,q}$ is that it is explicitly computable on the bottom layer due to functoriality of its construction. Restricting diagram \eqref{diag:gtol} to the underlying compact pairs we obtain the diagram
$$
\begin{CD}
(\liek,K)@<j_{K}<<(\liel,L)@>1>>(\liel,L)\\
@Ai_{\liek}AA @AAi_{\liek\cap \liel}A @AAi_{\liel}A\\
(\liek,C)@<j_{C\cap L}<<(\liel,C\cap D)@>j_{D}>>(\liel,D)\\
@Ai_{\lieq\cap\liek}AA @Ai_{\lieq\cap\liep}AA @AAi_{\liep\cap\liel}A\\
(\lieq\cap\liek,C)@<k_{C}<<(\lieq\cap\liep\cap\liel,C\cap D)@>k_{D}>>(\liep\cap\liel,D)\\
\end{CD}
$$
which again satisfies conditions \eqref{cond:sum}, \eqref{cond:int}, \eqref{cond:num} due to the $\theta$-stability of the involved parabolic subalgebras.

Now $\eta_{\lieq\cap\liep}$ induces a $(\lieq\cap\liep,C\cap D)$-map
$$
\eta_{\lieq\cap\liep}:\;\;\;\mathcal F_{k_{C}}(X)\;\to\;\mathcal F^\vee_{k_{D}}(Y),
$$
and we obtain by the construction in section \ref{sec:inducedmaps} a natural $(\liel,L)$-equivariant map
$$
\eta_{\liel,q}:\;
\mathcal F_{j_{K}}L_{S_\lieq+q}P_{i_\liek\circ i_{\overline{\lieq}\cap\liek}}(X)\otimes D_\liek^*
$$
$$
\;\to\;
L_{S_\liep+q}P_{i_\liel\circ i_{\overline{\liep}\cap\liek}}(Y)\otimes D_{\liel}^*.
$$
The bottom layer map is the canonical monomorphism
$$
\mathcal B_{\lieg,q}:\;\;\;
L_{S_\lieq+q}P_{i_\liek i_{\overline{\lieq}\cap\liek}}(X)\to 
\mathcal F_{(\liek,K)\to(\lieg,K)}
L_{S_\lieq+q}P_{i_\lieg i_{\overline{\lieq}}}(X),
$$
and similarly for $(\lieh,L)$. We remark that the bottom layer is natural in the base field $E$.
\begin{proposition}\label{prop:bottomlayer}
The restriction of $\eta_{\lieh,q}$ to the bottom layer equals $\eta_{\liel,q}$, and its image lies in the bottom layer again.
\end{proposition}

\begin{proof}
The statement of the proposition amounts to the commutativity of the diagram
$$
\begin{CD}
L_{S_\lieq+q}P_{i_\liek i_{\overline{\lieq}\cap\liek}}(X)\otimes D_\liek^*
@>{\mathcal B_{\lieg,q}\otimes 1_{D_\liek^*}}>>
\mathcal F_{(\liek,K)\to(\lieg,K)}L_{S_\lieq+q}P_{i_\lieg i_{\overline{\lieq}}}(X)\otimes D_\liek^*\\
@V{\eta_{\liel,q}}VV @VV{\eta_{\lieh,q}}V\\
L_{S_\liep+q}P_{i_\liel i_{\overline{\liep}\cap\liel}}(Y)\otimes D_{\liel}^*
@>{\mathcal B_{\lieh,q}\otimes 1_{D_\liel^*}}>>
\mathcal F_{(\liel,L)\to(\lieh,L)}L_{S_\liep+q}P_{i_\liel i_{\overline{\liep}}}(Y)\otimes D_{\liel}^*
\end{CD}
$$
This may be verified step by step, and the only non-trivial cases are the ones involving hard duality. Yet we know from the proof of Proposition 6.35 in \cite{book_knappvogan1995}, i.e.\ (6.43) and (6.44) of loc.\ cit.\ that hard duality for $(\liek,K)$ is compatible with hard duality for $(\lieg,K)$, and smiliarly for $(\liel,L)$ and $(\lieh,L)$. This concludes the proof.
\end{proof}

\subsubsection{Duality}\label{sec:duality}

In this subsection we assume that we are in the general situation of section \ref{sec:inducedmaps}, and we assume additionally that $Y$ is admissible, which implies its reflexivity as a $(\liep,C\cap L)$-module. The images of $Y$ under the cohomological induction and production functors are again admissible and in particular reflexive.

By easy and hard duality
$$
L_{S_\liep+q}P_{i_\liel i_{\overline{\liep}}}
(Y)^\vee\;=\;
L_{S_\liep-q}P_{i_\liel i_{\liep}}
(Y^\vee)\otimes D_{\liel}^*.
$$
Due to the reflexivity of this module, and the triviality of $D_{\Delta(\liel)}^*$, we have a natural isomorphism
$$
\Hom_{\lieh,L}(
L_{
S_\lieq+q}P_{i_\lieg i_{\overline{\lieq}}}
(X)\otimes D_{\liek}^*
,
L_{S_\liep+q}P_{i_\liel i_{\overline{\liep}}}
(Y)\otimes D_{\liel}^*
)
\;\to\;
$$
$$
\Hom_{\Delta(\liel),\Delta(L)}(
L_{S_\lieq+q}P_{i_\lieg i_{\overline{\lieq}}}
(X)\otimes D_{\liek}^*
\otimes
L_{S_\liep-q}P_{i_\liel i_{\liep}}
(Y^\vee)\otimes D_{\liel}^*,
D_{\Delta(\liel)}^*
),
$$
for $\Delta$ the diagonal embedding. Under this isomorphism, after a choice of basis of $D_{\liel}^*$, the map $\eta_{\lieh,q}$ corresponds to a $(\Delta(\liel),\Delta(L))$-equivariant functional
$$
\eta_{\Delta(\liel)}:\;\;
L_{S_\lieq+q}P_{i_\lieg i_{\overline{\lieq}}}(X)\otimes D_{\liek}^*
\otimes_E
L_{S_\liep-q}P_{i_\liel i_{\liep}}(Y^\vee)\otimes D_{\liel}^*
\;\to\;D_{\Delta(\liel)}^*.
$$
The reductive pair
$$
(\lieg\times\lieh,K\times L),
$$
the parabolic pair
$$
(\lieq\times\overline{\liep},C\times D),
$$
satisfy, with respect to the reductive pair
$$(\Delta(\liel),\Delta(L)),$$
and its parabolic subpair
$$(\Delta(\liel),\Delta(L)),$$
the conditions \eqref{cond:sum}, \eqref{cond:int} and \eqref{cond:num}. If we let $X\otimes Y^\vee$ replace $X$ and the trivial module replace $Y$, we may produce a functional $\eta_{\Delta(\liel),0}$, that, up to a choice of basis of $D_{\liel}^*$ and pullback along the K\"unneth map
$$
L_{S_\lieq+q}P_{i_\lieg i_{\overline{\lieq}}}(X)\otimes D_{\liek}^*
\otimes_E
L_{S_\liep-q}P_{i_\liel i_{\liep}}(Y^\vee)\otimes D_{\liel}^*\;\to\;
L_{S_\lieq+S_\liep}P_{i_{\lieg\times\liel} i_{\overline{\lieq}\times\liep}}
(X\otimes_E Y^\vee\otimes D_{\liek\times\liel}^*),
$$
equals $\eta_{\Delta(\liel)}$. This reduces us to the situation where $\liep=\liel$, $D=L$, $S_\liep=0$ and $q=0$ whenever $Y$ is admissible.

Remark that if
$$
L_{S_\lieq+q}P_{i_\lieg i_{\overline{\lieq}}}(X)\otimes D_{\liek}^*
\otimes_E
L_{S_\liep-q}P_{i_\liel i_{\liep}}(Y^\vee)\otimes D_{\liel}^*\;=\;0
$$
for $q\neq 0$ then the K\"unneth map is an isomorphism for $q=0$.

\subsubsection{Composability}

If we are given two commutative diagrams of the form \eqref{diag:gtol}, where the last column of the first diagram equals the first column of the second diagram, and the inducing data is the same for this column in both diagrams, we may consider the composition of the two morphisms constructed via the two diagrams after choices of bases for the dualizing modules. We remark that in general the resulting composite morphisms may in general not be produced directly via a single diagram of the form \eqref{diag:gtol}.

\subsection{Explicit formulae}

In this section we give an explicit description of $\eta_{\lieh,q}$ on the level of complexes. For our purpose it is sufficient to restrict our attention to $K$-equivariant complexes.

\begin{proposition}\label{prop:explicitddeltaalphad}
As a $(\liel,L)$-morphism, the map
$$
\mathcal D_{\lieg\cap\liel}^{-1}\circ\delta\circ \alpha_1\circ \mathcal D_{\lieg}
$$
is, up to normalization of the integrals (Haar measures in the case $E=\RR$) $dk\in\Int_E(K)$ and $dm\in\Int_E(L)$ on $K$ and $L$ respectively\footnote{The duality maps $\mathcal D_{\lieg\cap\liel}$ and $\mathcal D_\lieg$ both depend on choices for $dk$ and $dm$, cf.\ section \ref{sec:hardduality}.}, given on the level of complexes by the restriction map
$$
\OO_E(K)dk\otimes_{R(C)}
(\bigwedge^{S_\lieq+q}(\liek/\liec)\otimes_E
P_{i_{\overline{\lieq}}}(X))\otimes_E
\bigwedge^{2S_\lieq}(\liek/\liec)^*\to
$$
$$
\OO_E(L)dm\otimes_{R(C\cap D)}
(\bigwedge^{S_\liep+q}(\liel/\liec\cap\lied)\otimes_E
P_{i_{\overline{\lieq}}}(X))\otimes_E
\bigwedge^{T_{\lieq\cap\liep}}(\liel/\liec\cap\lied)^*.
$$
induced by the canonical map $\OO_E(K)\to \OO_E(L)$ and the natural map
$$
\bigwedge^{S_\lieq+q}(\liek/\liec)\otimes_E
\bigwedge^{2S_\lieq}(\liek/\liec)^*=
\bigwedge^{S_\lieq-q}(\liek/\liec)^*
\;\to\;
$$
$$
\bigwedge^{S_{\lieq}-q}(\liel/\liec\cap\lied)^*=
\bigwedge^{S_\liep+q}(\liel/\liec\cap\lied)\otimes_E\bigwedge^{T_{\lieq\cap\liep}}(\liel/\liec\cap\lied)^*.
$$
\end{proposition}

\begin{proof}
We identify $R_E(K)$ with $\OO_E(K)$ via our choice of integral (Haar measure) and similarly for the other reductive groups involved. On the level of complexes $\delta\circ \alpha_1$ is induced by the composition of the two maps
$$
\Hom_{C}(R_E(K),\Hom_E(U(\liek)\otimes_{U(\liec)}\bigwedge^q(\liek/\liec),U(\lieg)\otimes_{U(\overline{\lieq})} X)_{C})\to
$$
$$
\Hom_{C\cap D}(R_E(L),\Hom_E(U(\liek)\otimes_{U(\liec)}\bigwedge^q(\liek/\liec),U(\lieg)\otimes_{U(\overline{\lieq})} X)_{C})\to
$$
$$
\Hom_{C\cap D}(R_E(L),\Hom_E(U(\liel)\otimes_{U(\liec\cap\lied)}\bigwedge^q(\liel/\liec\cap\lied),U(\lieg)\otimes_{U(\overline{\lieq})} X)_{C\cap D}),
$$
where the first map is given by pullback along the natural map $R_E(L)\to R_E(K)$, the transpose of the map $\OO_E(K)\to\OO_E(L)$, and the second map is restriction along
$$
U(\liel)\otimes_{U(\liec\cap\lied)}\bigwedge^q(\liel/\liec\cap\lied)\;\to\;
U(\liek)\otimes_{U(\liec)}\bigwedge^q(\liek/\liec\cap\liel).
$$
The natural map
$$
U(\liel)\otimes_{U(\liec\cap\lied)}\bigwedge^q(\liel/\liec\cap\lied)^*\otimes_E U(\lieg)\otimes_{U(\overline{\lieq})} X\;\;\to\;\;
$$
$$
\Hom_E(U(\liel)\otimes_{U(\liec\cap\lied)}\bigwedge^q(\liel/\liec\cap\lied),U(\lieg)\otimes_{U(\overline{\lieq})} X)_{C\cap D}
$$
is an isomorphism, and analogously for the term corresponding to $\liek$. Application of hard duality first to complex given by the inner $\Hom_E$ under application of Proposition 5.90 of loc.\ cit., and then to the outer $\Hom_{C}$ and $\Hom_{C\cap D}$ concludes the proof.
\end{proof}

The map $\beta_3$  in the construction of $\eta_{\lieh,q}$ is the natural isomorphism
\begin{equation}
U(\liel)\otimes_{U(\overline{\lieq\cap\liep})} X
\to
U(\lieg)\otimes_{U(\overline{\lieq})} X,\;\;\;l\otimes x\mapsto l\otimes x,
\label{eq:beta3}
\end{equation}
and $\beta_4$ is given by the natural map
$$
U(\liel)\otimes_{U(\overline{\lieq\cap\liep})} Y
\to
U(\liel)\otimes_{U(\overline{\liep})} Y,\;\;\;l\otimes y\mapsto l\otimes y.
$$

On the level of complexes the map $\beta_2$ is given by the projection map
$$
R(L)\otimes_{R(C\cap D)}
(\bigwedge^{S_\liep+q}(\liel/\liec\cap\lied)\otimes_E
P_{i_{\overline{\liep}}}(Y))\;\to\;
$$
$$
R(L)\otimes_{R(D)}
(\bigwedge^{S_\liep+q}(\liel/\lied)\otimes_E
P_{i_{\overline{\liep}}}(Y)).
$$

Let us assume for a moment that $\liep=\liel$, then $D=L$. By section \ref{sec:duality} this is not a serious restriction, but simplifies the formulation considerably. Under these assumptions $S_\liep=0$ and thus only $q=0$ is of interest. Then
$$
(L_{0}P_{i_{\lieg\cap\liel}})
\mathcal F_{j_{D}}^\vee P_{i_{\overline{\liep}}}(Y)\;=\;
R(L)\otimes_{R(C\cap L)} Y,
$$
and
$$
L_{0}P_{i_{\liel}\circ i_{\overline{\liep}}}(Y)\;=\;Y.
$$
Since $D=L$, $Y$ is a rational $L$-module and decomposes into isotypic components
$$
Y\;=\;\bigoplus_j Y_j,
$$
and to each $Y_j$ corresponds a unique irreducible (but not necessarily absolutely irreducible) rational $L$-module $Z_j$ over $E$. To each such $Z_j$ then correponds a unique cosimple coalgebra $C_j\subseteq\OO_E(L)$, and an approximate unit $1_{C_j}\in R_E(L)$ that we simply denote by $1_j$. Each approximate identity $1_j$ acts as identity on $Z_j$ and hence on $Y_j$ and as zero on all other isotypic components.

Then each element $y\in R_E(L)\otimes_{R(C\cap L)} Y$ may be represented as a finite sum
\begin{equation}
y = \sum_{k}r_k\otimes y_k,\;\;\;\;\;\;r_k\in R_E(L),\;y_k\in Y_{j(k)},
\label{eq:ysum}
\end{equation}
for some index $j(k)$ depending on $k$.

\begin{proposition}[Counit formula]\label{prop:beta2}
If $\liep=\liel$, then $\beta_2$ is the counit
$$
P_{i_{\lieg\cap\liel}}\circ\mathcal F_{i_{\lieg\cap\liel}}^\vee
\;\to\;
{\bf1},
$$
i.e.\ for any $y\in R_E(L)\otimes_{R(C\cap L)}Y$ we have the explicit formula
$$
\beta_2(y)\;=\;
\sum_{k}r_k\cdot y_k\;=\;\sum_k(r_k\cdot 1_{j(k)})\cdot y_k\;\in\;Y
$$
in the notation of \eqref{eq:ysum}.
\end{proposition}

\begin{proof}
This is formal.
\end{proof}

\begin{proposition}\label{prop:expliciteta}
If $\liep=\liel$, then on the level of complexes the map $\eta_{\lieh,0}$ is given, as an $(\liel,L)$-morphism, by the map
$$
\OO_E(K)\otimes_{R(C)}
(
\bigwedge^{S_\lieq}(\liek/\liec)\otimes
U(\lieg)\otimes_{U(\overline{\lieq})}X)\otimes
\bigwedge^{2S_\lieq}(\liek/\liec)^*
\;\to\;Y
,
$$
$$
f\otimes (u\otimes l\otimes x)\otimes \varepsilon
\;\;\mapsto\;\;
\pi^*(\varepsilon(u\wedge -))\cdot f|_L\cdot (l\cdot\eta_{\lieq\cap\liep}(x))
$$
with the elements in the corresponding factors, $f|_L\in\OO_E(L)$ the canonical image of $f\in\OO_E(K)$ under the restriction map $\OO_E(K)\to\OO_E(L)$, and $l\in U(\liel)\subseteq U(\lieg)$.
\end{proposition}

We remark that we abuse the notation by describing the map without explicitly taking the relative tensor products over $R(C)$ and $U(\overline{\lieq})$ into account.

\begin{proof}
By Proposition \ref{prop:explicitddeltaalphad} this follows from our preceeding discussion.
\end{proof}

If $Y$ is one-dimensional, the sum \eqref{eq:ysum} collapses a single summand, and we obtain

\begin{corollary}\label{cor:expliciteta}
If $\liep=\liel$ and $Y$ is one-dimensional, then for a suitable choice of $dk\in\Int_{E^+}(K)$ and $dl\in\Int_{E^+}(L)$ the map $\eta_{\lieh,0}$ is on the level of complexes explicitly given by
$$
f\otimes (u\otimes l\otimes x)\otimes\varepsilon
\;\mapsto\;
\pi^*(\varepsilon(u\wedge -))\cdot
(f|_L\cdot 1_Y)
\cdot
(l\cdot\eta_{\lieq\cap\liep}(x)),
$$
where $1_Y\in R_E(L)$ denotes the approximate identity corresponding to the isomorphism class of $Y$.
\end{corollary}

\subsection{A non-vanishing criterion}\label{sec:nonvanishingcriterion}

The following Theorem is a generalization of Proposition 2.9 in \cite{sunpreprint2}.

\begin{theorem}\label{thm:nonvanishing}
In the general situation, i.e.\ with $(\lieg,K)$, $(\lieh,L)$, $\lieq$, $\liep$, satisfying \eqref{cond:sum}, \eqref{cond:int}, and \eqref{cond:num}, and $X$ arbitrary and $Y$ admissible, and assume that $X\otimes\bigwedge^{\dim\lieu}\overline{\lieu}$ and $Y\otimes\bigwedge^{\dim\liev}\overline{\liev}$ have infinitesimal characters within the fair range respectively. Then the map
$$
\sum_{q\in\ZZ} \eta_{\lieh,q}\;:\;\;\;
\bigoplus_{q\in\ZZ}
\mathcal F_{j_{K}}L_{S_\lieq+q}P_{i_\liek\circ i_{\overline{\lieq}\cap\liek}}(X)
\otimes D_\liek^*
$$
$$
\;\to\;
\bigoplus_{q\in\ZZ}
L_{S_\liep+q}P_{i_\liel\circ i_{\overline{\liep}\cap\liek}}(Y)\otimes D_\liel^*.
$$
is non-zero on the bottom layer if and only if there exists an $x\in X$ with the following three properties:
\begin{itemize}
\item[(i)] $x$ generates an irreducible $C$-subrepresentatation $\xi\subseteq X$,
\item[(ii)] $\eta_{\lieq\cap\liep}(x)$ generates a non-zero $D$-subrepresentation $\varrho\subseteq Y$,
\item[(iii)] the natural map
$$
\Hom_{L}(
L_{S_\liep}P_{i_\liel i_{\overline{\liep}\cap\liel}}\varrho,
\mathcal F_{L\to K}L_{S_\lieq}P_{i_\liek i_{\overline{\lieq}\cap\liek}}\xi)
\;\to\;
$$
$$
\Hom_E(
H^0(\liev\cap\liel;
L_{S_\liep}P_{i_\liel i_{\overline{\liep}\cap\liel}}\varrho),
H_0(\overline{\lieu}\cap\liek;L_{S_\lieq}P_{i_\liek i_{\overline{\lieq}\cap\liek}}\xi))
$$
is non-zero.
\end{itemize}
\end{theorem}

Another way to phrase condition (iii) is the following.
Remark that
$$
Z\;:=\;1_Y\cdot \mathcal F_{L\to K}L_{S_\lieq}P_{i_\liek i_{\overline{\lieq}\cap\liek}}\xi
$$
is the union of the images of all $L$-equivariant maps
$$
L_{S_\liep}P_{i_\liel i_{\overline{\liep}\cap\liel}}\varrho
\;\;\;\to\;\;\;
\mathcal F_{L\to K}L_{S_\lieq}P_{i_\liek i_{\overline{\lieq}\cap\liek}}\xi.
$$
Then (iii) is equivalent to saying that the natural map
$$
H^0(\liev\cap\liel;Z)\;\to\;
H_0(\overline{\lieu}\cap\liek;L_{S_\lieq}P_{i_\liek i_{\overline{\lieq}\cap\liek}}\xi)
$$
has non-zero image.

\begin{proof}
Proposition \ref{prop:bottomlayer} reduces us to the situation where $\lieg=\liek$, $\liel=\liel$, and the compatibility with duality (cf.\ section \ref{sec:duality}), reduces us further to the case $\liep=\liel$, $D=L$, $q=0$, $S_\liep=0$, and one-dimensional $Y$, as the statement of the theorem is compatible with duality via the natural isomorphism
$$
H_0(\liev\cap\liel;L_{S_\liep}P_{i_\liel i_{\liep}}(Y^\vee))\;\cong\;
H^0(\liev\cap\liel;D_{\liel}\otimes L_{S_\liep}P_{i_\liel i_{\overline{\liep}}}(Y))^\vee.
$$
Then
$$
S_\lieq\;=\;T_{\lieq\cap\liep}\;=\;\dim\liel/\liec\;=\;\dim\lieu.
$$
Then $D_\liel^*=E$ canonically for trivial reasons and $D_\liek^*$ is a trivial $(\lieg,C)$-module, that we may consider as a trivial $(\lieg,K)$-module without loss of generality, and as a trivial $(\lieq,C)$-module as well.

The adjointness relation for $\mathcal F_{j_K}$ and $I_{j_K}$ and the algebraic Borel-Bott-Weyl Theorem (cf.\ Theorem \ref{thm:borelbottweil}), the $\Hom$-set
$$
\Hom_{L}(
\mathcal F_{L\to K}
L_{S_\lieq}P_{i_\liek i_{\overline{\lieq}}}
X\otimes D_\liek^*,
Y)
$$
is naturally isomorphic to
\begin{equation}
\begin{CD}
\Hom_{K}(
L_{S_\lieq}P_{i_\liek i_{\overline{\lieq}}}X\otimes D_\liek^*,
I_{j_K}
Y)@>\sim>>
\end{CD}
\label{iso:bbw}
\end{equation}
$$
\Hom_{C}(
X\otimes D_\liek^*,
H^{S_\lieq}(\overline{\lieu};
I_{j_K}
Y)).
$$
The isomorphism \eqref{iso:bbw} is on the level of complexes induced by the map which associates to
$$
f:\;\;\;
R_E(K)\otimes_{R_E(C)}
\left(
\bigwedge^{S_\lieq}(\liek/\liec)\otimes_E
U(\liek)\otimes_{U(\overline{\lieq})} X
\right)\otimes D_\liek^*
\;\to\;
R_E(K)\otimes_{R_E(L)}
Y
$$
the map
$$
r(f)\;:\;\;\;X\otimes D_\liek^*\;\to\; \Hom_E(\bigwedge^{S_\lieq}\overline{\lieu},\;R_E(K)\otimes_{R_E(L)}
Y),
$$
$$
x\otimes\varepsilon\;\;\;\mapsto\;\;\;\left[\;u\;\mapsto\;
\sum_{
Z\in\widehat{C}_E
}
f(1_{L_{S_\lieq}P_{i_\liek i_{\overline{\lieq}}}(Z)}\otimes
u\otimes 1\otimes
(1_Z\cdot x)\otimes\varepsilon)\;\right].
$$
Note that the approximate identity $1_{L_{S_\lieq}P_{i_\liek i_{\overline{\lieq}}}(Z)}$ is well defined, since $L_{S_\lieq}P_{i_\liek i_{\overline{\lieq}}}(Z)$ is either irreducible or $0$ by the Borel-Bott-Weyl Theorem. This also shows that there are only finitely many non-zero summands. Furthermore by Poincar\'e duality the cohomology class associated to $r(f)(x\otimes\varepsilon)$ is non-zero if and only if for some $Z\in\widehat{C}_E$ the element
$$
f(1_{L_{S_\lieq}P_{i_\liek i_{\overline{\lieq}}}(Z)}\otimes
u\otimes 1\otimes
(1_Z\cdot x)\otimes\varepsilon)\;\in\;
$$
$$
L_{S_\lieq}P_{i_\liek i_{\overline{\lieq}}}(Z)\otimes_E
(L_{S_\lieq}P_{i_\liek i_{\overline{\lieq}}}(Z)^\vee\otimes_L Y)
\;\subseteq\;
R_E(K)\otimes_{R_E(L)}Y
$$
contributes to
$$
H_0(\overline{\lieu};L_{S_\lieq}P_{i_\liek i_{\overline{\lieq}}}(Z)).
$$

By Corollary \ref{cor:expliciteta} the map corresponding to $\eta_{\liel,0}$ under $r$ is the map that sends
$$
x\otimes \varepsilon\;\in\;
D_\liek^*\otimes X
$$
to
$$
u\;\mapsto\;
\pi^*(\varepsilon(u\wedge-))\cdot\!\!
\sum_{Z\in\widehat{C}_E}
1_{L_{S_\lieq}P_{i_\liek i_{\overline{\lieq}}}(Z)}\otimes
(1_{L_{S_\lieq}P_{i_\liek i_{\overline{\lieq}}}(Z)}|_L \cdot 1_Y)
\cdot
\eta_{\lieq\cap\liep}(1_Z\cdot x).
$$
For $u\neq 0$ and $\varepsilon\neq 0$ we have always $\pi^*(\varepsilon(u\wedge-))\neq 0$. Therefore the non-vanishing of the above cohomology class is equivalent to the existence of an $x\in X$ with $\eta_{\lieq\cap\liep}(x)\neq 0$, generating an irreducible $C$-representation $\xi\subseteq X$ with the property that the $L$-submodule
$
1_Y\cdot \mathcal F_{L\to K}L_{S_\lieq}P_{i_\liek i_{\overline{\lieq}}}(\xi)
$
contributes non-trivially to $H_0(\overline{\lieu};L_{S_\lieq}P_{i_\liek i_{\overline{\lieq}}}(\xi))$.
\end{proof}

\section{Algebraic setup}

In this section we specialize the general notation from the preceeding sections for our application to special values of $L$-functions.

\subsection{Algebraic groups}

For any natural number $m$ we set
$$
G_{m}\;=\;\res_{F/\QQ}(\GL_{m}).
$$
Then $G_m$ is a quasi-split reductive linear algebraic group over $\QQ$.

In the remaining sections we fix a natural number $n\geq 1$ and consider
$$
G\;:=\;G_{2n},
$$
and denote by
$$
H\;\subseteq\; G
$$
a copy of $G_n\times G_n$ embedded via the diagonal embedding
$$
H\to G,\quad(h_1,h_2)\mapsto \begin{pmatrix}h_1&\\&h_2\end{pmatrix}.
$$

In the sequel we make frequent use of the following convention. A (quasi-)character $\chi$ of $G_n$ or of $G_n(A)$ for a $\QQ$-algebra $A$ gives rise to a (quasi-) character of $H$ via the rule
\begin{equation}
G_n\times G_n\ni (h_1,h_2)\;\mapsto\;\mathcal \chi(h_2^{-1}h_1),
\label{eq:characterpullback}
\end{equation}
i.e.\ via composition with the map
$$
G_n\times G_n\to G_n,\quad (h_1,h_2)\mapsto h_2^{-1}h_1.
$$
We call a (quasi-)character of $H$ resp.\ $H(A)$ which arises in this way {\em admissible}.

We have a decomposition into a quasi-product
\begin{equation}
G_1\;=\;\res_{F/\QQ}(\GL_1)\;=\;G^{\rm s}_1\cdot G^{\rm an}_1,
\label{eq:rationaltorusdecomposition}
\end{equation}
where $G^{\rm s}_1$ is the maximal $\QQ$-split torus and $G^{\rm an}_1$ is the maximal $\QQ$-anisotropic subtorus in $G_1$. The latter is a quasi-complement of $G^{\rm s}_1$ in $G_1$, i.e.
$$
G^{\rm s}_1\cap G^{\rm an}_1\quad\text{is finite}.
$$
The projection
$$
p_{\rm s}:\quad G_1\to G_1/G^{\rm an}_1\cong\GL_1
$$
corresponds to the Norm character
$$
N_{F/\QQ}:\quad F^\times\to\QQ^\times,
$$
and its composition with the determinant induces a character
$$
\mathcal N:=p_{\rm s}\circ\res_{F/\QQ}\det:\quad G_n\to\GL_1.
$$
We consider $\mathcal N$ as a character of $H$ via the rule \eqref{eq:characterpullback}, i.e.
$$
h\;\mapsto\;
\mathcal N(h_2^{-1}h_1),
$$
that we simply call the {\em norm character} on $H$. To describe its behavior on the real points, we introduce for each archimedean place $v$ of $F$ the local {\em sign character}
$$
\sgn_v:\quad \GL_n(F_v)\;\to\;\RR^\times,\quad h_v\;\mapsto\; \frac{\det(h_v)}{|\det(h_v)|_v}.
$$
By abuse of notation we also consider $\sgn_v$ and the norm $|\cdot|_v$ (implicitly composed with the determinant) as a character of $G_n(\RR)$. Recall that $F$ is totally real. Then on real points we have the sign character
$$
\sgn_\infty:=\otimes_v \sgn_v:\quad G_n(\RR)\;\to\;\RR^\times,
$$
and the archimedean norm character
$$
|\cdot|_\infty:=\otimes_v |\cdot|_v:\quad G_1(\RR)\to\RR^\times.
$$
We obtain that for all $h\in H(\RR)$
\begin{equation}
\mathcal N(h)\;=\;
\sgn_\infty(h_2^{-1}h_1)\cdot|h_2^{-1}h_1|_\infty.
\label{eq:algebraicarchimedeannorm}
\end{equation}
We remark that the group $\Hom(G_n(\RR),\CC^\times)$ of quasi-characters of $G_n(\RR)$ is, as a complex manifold, a disjoint union of
$$
2^{r_F}\;=\;\#\pi_0(G_n(\RR))
$$
copies of $\CC$. Each component corresponds uniquely to a finite order character of $G_n(\RR)$. Each such character is of the form
$$
\sgn_\infty^\delta:=\otimes_v \sgn_v^{\delta_v}:\quad G_n(\RR)\to\CC^\times,
$$
with
$$
\delta=(\delta_v)_v\;\in\;\prod_v\{0,1\}.
$$
Then the component of $\sgn_\infty^\delta$ is parametrized by the charts
\begin{equation}
\CC\ni s\;\mapsto\;\omega^\delta_s\;:=\;\sgn_\infty^\delta\otimes|\cdot|_\infty^s\in\Hom(G_n(\RR),\CC^\times).
\label{eq:quasicharacterchart}
\end{equation}
We call $\omega_s^\delta$ {\em algebraic} whenever $s\in\ZZ$. If $\chi$ is a quasi-character of $G_n(\RR)$, we set for $k\in\ZZ$
$$
\chi[k]\;:=\;\chi\otimes\left(\mathcal N^{\otimes k}\right).
$$
If $\chi$ is algebraic, then so is $\chi[k]$. As before we consider these characters as admissible characters of $H(\RR)$ as well.

We fix a non-trivial continuous character $\psi:F\backslash\Adeles_F\to\CC^\times$, and all Gau\ss{} sums we consider are understood with respect to $\psi$, and normalized in such a way that when $\chi=\chi'\otimes|\cdot|_{\Adeles_F}^{k(\chi)}$, with $\chi'$ of finite order, then
$$
G(\chi)\;:=\;
G(\chi')\;:=\;
\sum_{x{\rm mod}\mathfrak{f}_{\chi'}}
\chi'\left(\frac{x}{f_{\chi'}}\right)
\psi\left(\frac{x}{f_{\chi'}}\right),
$$
where $\mathfrak{f}_{\chi'}=\mathcal O_F\cdot f_{\chi'}$ is the conductor of the finite order character $\chi':F^\times\backslash\Adeles_F^\times$.

\subsection{Rational models of compact groups}

For each $m$ we fix a $\QQ$-form $K_m\subseteq G_m$ of the standard maximal compact subgroup of $G_m(\RR)$, as in section 1.2 of \cite{januszewskipart1}. We recall that
$$
K_m\;=\;\res_{F/\QQ}\Oo(m)
$$
with the \lq{}standard form\rq{} of $\Oo(m,F\otimes\RR)$ over $F$, corresponding to the Cartan involution
$$
\theta_m:\quad \res_{F/\QQ}\GL_m\to \res_{F/\QQ}\GL_m,\quad \res_{F/\QQ}\left(g\mapsto {}^tg^{-1}\right).
$$
Then $K_m$ is split over $E_K:=\QQ(\sqrt{-1})$ and
\begin{equation}
K_m(\RR)\;=\;
\prod_{v\mid\infty}\Oo(m,F_v)
\label{eq:klocaldecomposition}
\end{equation}
in a natural way.

We remark that the connected component $K_m^0$ of the identity of $K_m$ is geometrically connected and defined over $\QQ$. Furthermore we have a natural isomorphism
$$
\pi_0(K_m)\;=\;\pi_0(K_m(\RR)),
$$
where $K_m(\RR)$ on the right hand side is considered as a real Lie group. In particular we have
$$
\pi_0(K_m)\;=\;\pi_0(K_m(\RR))\;=\;\{\pm1\}^{r_F}.
$$
The inclusion $G_m\to G_{2m}$ induces a natural isomorphism
$$
\pi_0(K_m)\cong\pi_0(K_{2m}).
$$

We choose $K:=K_{2n}$ as a $\QQ$-form of the standard maximal compact subgroup of $G(\RR)$. Then the intersection of $K$ with $H$ is a $\QQ$-form $L$ of a standard maximal compact subgroup of $H(\RR)$. It is isomorphic to the product $K_n\times K_n$ and split over $E_K$.

We remark that we have a commutative square
$$
\begin{CD}
\pi_0(K)@>\sim>>\{\pm1\}^{r_F}\\
@AAA @AAA\\
\pi_0(L)@>\sim>>\{\pm1\}^{r_F}\times\{\pm1\}^{r_F}
\end{CD}
$$
where the vertical arrows are epimorphisms and the right vertical map is dual to the diagonal embedding. Passing dually to admissible characters of the component groups, we obtain isomorphisms.

We have a natural isomorphism between the group of quasi-characters $\Hom(G_n(\RR),\CC^\times)$ and the group of
one-dimensional representations of $(\lieg_n,K_n)_\CC$, sending a quasi-character $\chi$ to its derivative
$$
d\chi:\quad \lieg_{n,\CC}\to\CC,
$$
and to the complexification of its restriction to $K_n(\RR)$,
$$
\chi|_{K_n(\RR),\CC}:\quad K_n(\CC)\to\CC^\times.
$$
For the sake of notational simplicity we write $\omega_s^\delta$ also for the image of $\omega_s^\delta$ under this correspondence. Then, as $\mathcal N$ is defined over $\QQ$, and $\sgn_\infty^\delta$, as a rational character of $K_n$, is defined over $\QQ$ as well, we see with \eqref{eq:algebraicarchimedeannorm}, that for each $s\in\ZZ$ the algebraic quasi-character
$$
\omega_s^\delta\;=\;\sgn_\infty^\delta\otimes(\sgn_\infty\otimes\eta)^s,
$$
considered as a one-dimensional $(\lieg_n,K_n)$-module, is defined over $\QQ$. We write $X^{\rm alg}(G_n(\RR))$ for the group of algebraic quasi-characters of $G_n(\RR)$ resp.\ $(\lieg_n,K_n)$.

Recall that we defined $E_K=\QQ(\sqrt{-1})$. We write $\tau_K\in\Gal(E_K/\QQ)$ for the unique non-trivial automorphism, which we simply call {\em complex conjugation}.

\subsection{Finite-dimensional representations}

As $G$ is quasi-split we may find a Borel $P\subseteq G$ defined over $\QQ$. We choose a maximal torus $T\subseteq P$ over $\QQ$ and denote by $P^-$ the opposite of $P$. We let $X(T)$ denote the characters of $T$ defined over $\overline{\QQ}$ or $\CC$, which amounts to the same, and let $X_E(T)$ denote the subgroup of characters defined over $E$. Then the absolute Galois group $\Gal(\overline{\QQ}/\QQ)$ resp.\ $\Aut(\CC/\QQ)$ act naturally on $X(T)$ from the right via the rule
$$
\mu^\tau(t)\;=\;\mu\left(t^{\tau^{-1}}\right)^{\tau},
$$
for $\mu\in X(T)$, $t\in T(\CC)$ and $\tau\in\Aut(\CC/\QQ)$. If $M_\mu$ denotes the irreducible rational $G$-module with highest weight $\mu\in X(T)$ over $\CC$ (or $\overline{\QQ}$), then the Galois twisted module
\begin{equation}
(M_\mu)^\tau\;:=\;M_\mu\otimes_{\CC,\tau^{-1}} \CC
\label{eq:galoistwistedmodule}
\end{equation}
is irreducible of highest weight $\mu^\tau$. Furthermore it is easy to see that $M_\mu$ is defined over the field of rationality $\QQ(\mu)$ of $\mu$.

We remark that this Galois action on complex rational representations of $G$ is compatible with our rational structure on the category of complex $(\lieg,K)$-modules.

\subsection{Arithmeticity conditions}\label{sec:artihmeticity}

Following the terminology of \cite[Section 3]{januszewski2015}, adapted to the Shalika case, we call an absolutely irreducible rational $G$-module $M$ over $E/\QQ$ {\em arithmetic}, if it is essentially self-dual over $\QQ$, i.e.
$$
M^{\vee}\;\cong\;M\otimes\xi,
$$
with a $\QQ$-rational character $\xi\in X_\QQ(G)$. We remark that as $F$ is totally real, all its Galois twists $M^\tau$, $\tau\in\Gal(\overline{\QQ}/\QQ)$ are arithmetic if $M$ is so.

We identify the center $Z$ of $G=G_{2n}$ with a copy of $G_1$ via the isomorphism
\begin{equation}
G_1\to Z,\quad a\mapsto a\cdot{\bf1}_{2n}.
\label{eq:centeriso}
\end{equation}
The center $Z$ lies naturally in our maximal torus $T\subseteq G$, and we have a decomposition
$$
T\;=\;T^{\rm der}\cdot Z,
$$
with
$$
T^{\rm der}\;=\;T\cap G^{\rm der}.
$$
The intersection of the latter with $Z$ is finite. The decomposition \eqref{eq:rationaltorusdecomposition} of $G_1$ induces a projection
$$
p_T:\quad 
T\to T/(T^{\rm der}\cdot G_1^{\rm an})=:Z^{\rm s}\;\cong\;\GL_1,
$$
and we obtain a monomorphism
$$
p_T^*:\quad X_\QQ(Z^{\rm s})\to X_\QQ(T).
$$
We fix the identification
$$
Z^{\rm s}\;\cong\;\GL_1,
$$
subject to the same positivity condition as the isomorphism \eqref{eq:centeriso}. This induces an identification
$$
X_\QQ(Z^{\rm s})\;=\;
\ZZ,
$$
and we simply write $w_1$ for the image of $w_1\in X_\QQ(Z^{\rm s})$ under $p_T^*$.

If $M$ is a rational $G$-module of highest weight $\mu$, we denote by $\mu^{\vee}$ the highest weight of $M^{\vee}$. Then $M$ is arithmetic if and only if
$$
\mu^\vee-\mu\;\in\;\image(p_T^*).
$$
We have
\begin{equation}
(\mu^\sigma)^{\vee}-\mu^\sigma
\;=\;
w_1,
\label{eq:prearithmeticitiyweights}
\end{equation}
with $w_1$ independent of $\sigma\in\Gal(\overline{\QQ}/\QQ)$. Following Clozel \cite{clozel1990}, we may think of $w_1$ as a weight in the motivic sense.

A $\QQ$-rational character $\xi\in X_\QQ(H)\cong\ZZ^2$ of the form
\begin{equation}
\xi\;=\;(k,-w_1-k),\quad k\in\ZZ,
\label{eq:shalikacriticalxi}
\end{equation}
for $w_1$ as in \eqref{eq:prearithmeticitiyweights}, is called {\em critical} for $M$ if
$$
\Hom_H(M,\xi)\;\neq\;0.
$$
Due to the multiplicity one property of $\GL_n\times\GL_n$-equivariant functionals on irreducible $\GL_{2n}$-modules (cf.\ the non-compact analogue of Proposition \ref{prop:bonedimensionalfunctionals} below), we have
\begin{equation}
\dim_{\QQ(\mu)}\Hom_H(M_{\QQ(\mu)},\xi_{\QQ(\mu)})\;=\;1,
\label{eq:criticaldimensionone}
\end{equation}
for each critical $\xi$. We call $M$ {\em critical} if it admits a critical character $\xi$.



\section{Shalika models}\label{sec:shalika}

In this section we recall the theory of Friedberg-Jacquet \cite{friedbergjacquet1993,jacquetshalika1990} in the context of Shalika models for $\GL(2n)$. For a nice exposition of the subject, we refer to Section 3 in \cite{grobnerraghuram2014}.

\subsection{Group setup}

As in the preceeding section we put ourselves in the situation where
$$
G\;=\;\res_{F/\QQ}\GL_{2n},
$$
whose center we denote as before by $Z$, and recall that
$$
H\;=\;\res_{F/\QQ}(\GL_n\times\GL_n)\;\subseteq\;G,
$$
and $H$ is diagonally embedded in $G$. We factor $H$ into two copies of $\res_{F/\QQ}\GL_n$ accordingly, i.e.\
$$
H\;=\;H_1\times H_2,
$$
where $H_1$ denotes the upper block. Then the embedding of $H$ into $G$ is given explicitly by
$$
H_1\times H_2\to G,\quad (h_1,h_2)\mapsto \begin{pmatrix}h_1&\\&h_2\end{pmatrix}.
$$
Recall that $L=K\cap H$, then we have analogously
$$
L\;=\;L_1\times L_2
$$
with $L_j=K\cap H_j$.

The Shalika group is given by
$$
{\mathcal S}\;:=\;\res_{F/\QQ}
\left\{
\begin{pmatrix}
h & X\\
0& h
\end{pmatrix}
\in \GL_{2n}\;\mid\;h\in\GL_n,X\in M_n
\right\}\;\subseteq\;G_{2n}
$$
where $M_n$ denotes the $n\times n$ matrices. We consider ${\mathcal S}$ as a linear algebraic group over $\QQ$.

We fix factorizable (right) Haar measures on all adelized groups and compatibly on all their local factors.

\subsection{Automorphic representations}

We consider an irreducible cuspidal automorphic representation $\Pi$ of $G(\Adeles)$ with central character $\omega_\Pi$. We assume that we find an $n$-th root of $\omega_\Pi$ in the set of quasi-characters, and we fix such an $n$-th root, i.e.\ we fix a quasi-character $\eta:F^\times\backslash\Adeles_F^\times\to\CC^\times$ with the property that
$$
\omega_\Pi\;=\;\eta^n.
$$
Our choice of continuous character $\psi:F\backslash\Adeles_F\to\CC^\times$ together with $\eta$ induces a character
$$
\eta\otimes\psi:\quad {\mathcal S}(\Adeles)\to\CC^\times,
$$
$$
\begin{pmatrix}
h & 0\\
0& h
\end{pmatrix}
\cdot
\begin{pmatrix}
1 & X\\
0& 1
\end{pmatrix}
\;\mapsto\;\eta(\det(h))\psi(\tr(X)),
$$
and we may consider the intertwining operator
$$
{\mathcal S}_\psi^\eta:\quad\Pi\;\to\;\ind_{{\mathcal S}(\Adeles)}^{G(\Adeles)}(\eta\otimes\psi),
$$
$$
\varphi\;\mapsto\;
\left[
g\;\mapsto\;
\int_{Z(\Adeles){\mathcal S}(\QQ)\backslash{\mathcal S}(\Adeles)}
(\Pi(g)\cdot\varphi)(s)(\eta\otimes\psi)(s^{-1})ds
\right].
$$
We say that $\Pi$ admits an $(\eta,\psi)$-Shalika model if ${\mathcal S}_\psi^\eta$ is non-zero, i.e.\ if $\Pi$ admits an embedding into the representation induced from $(\eta,\psi)$. The following Theorem is due to Jacquet and Shalika.
\begin{theorem}[Theorem 1 in \cite{jacquetshalika1990}]\label{thm:jacquetshalika}
With the notation above, the following two statements are equivalent:
\begin{itemize}
\item[(i)] $\Pi$ admits an $(\eta,\psi)$-Shalika model.
\item[(ii)] For an appropriate finite set $S$ of places of $\QQ$, the twisted partial exterior square $L$-function
$$
L^S(s,\Pi,\wedge^2\otimes\eta^{-1})\;=\;
\prod_{v\not\in S}L(s,\Pi_v,\wedge^2\otimes\eta_v^{-1})
$$
has a pole at $s=1$.
\end{itemize}
\end{theorem}
In this theorem, and in the context of Shalika models in the sequel, we consider $\Pi$ as a representation of $G(\Adeles)$ and factor it accordingly into local representations. The same applies to the characters $\eta$ and $\psi$, they are considered over $\QQ$ via restriction of scalars.

We also mention the following useful characterization of representations admitting a Shalika model.
\begin{proposition}[Proposition 3.1.4 in \cite{grobnerraghuram2014}]
With the notation above, the following statements are equivalent:
\begin{itemize}
\item[(i)] $\Pi$ admits an $(\eta,\psi)$-Shalika model.
\item[(ii)] $\Pi$ is the transfer of a globally generic irreducible cuspidal representation of $\GSpin_{2n+1}(\Adeles_F)$ to $\GL_{2n}(\Adeles_F)$.
\end{itemize}
\end{proposition}
We remark that this characterization, as well as the one given in Theorem \ref{thm:jacquetshalika}, are valid more generally over any number field $F$.

In the case of $n=1$, we may choose $\omega_\Pi=\eta$ and the $(\eta,\psi)$-Shalika model then coincides with the Whittaker model and the theory discussed here specializes to the Rankin-Selberg theory in that case. Our results on period relations then are identical with those from \cite{januszewskipart1}.

From now on we assume $\Pi$ to admit an $(\eta,\psi)$-Shalika model and set
$$
\mathscr{S}(\Pi,\eta,\psi)\;:=\;{\mathcal S}_\psi^\eta(\Pi),
$$
and apply this notation also locally, i.e.\
$$
\mathscr{S}(\Pi_v,\eta_v,\psi_v)\;:=\;{\mathcal S}_{\psi_v}^{\eta_v}(\Pi_v),
$$
where ${\mathcal S}_\psi^\eta$ gives rise to the local intertwining operator
$$
{\mathcal S}_{\psi_v}^{\eta_v}:\quad \Pi_v\to\ind_{\mathcal S(\QQ_v)}^{G(\QQ_v)}(\eta_v\otimes\psi_v).
$$
We remark that vectors in the Shalika model satisfy the symmetry property
\begin{equation}
\mathcal S_\psi^\eta(\varphi)(sg)\;=\;
\eta(s\otimes\psi)(s)\cdot
\mathcal S_\psi^\eta(\varphi)(g),\quad g\in G(\Adeles),\,s\in\mathcal S(\Adeles).
\label{eq:shalikasymmetry}
\end{equation}
Furthermore, following \eqref{eq:characterpullback}, each quasi-character $\chi:H_1(\Adeles)\to\CC^\times$ gives rise to an {\em admissible} quasi-character
\begin{equation}
\begin{cases}
\chi:\quad &H(\Adeles)\to\CC^\times,\\
&h\;\mapsto\;
\chi(h_2^{-1}h_1),
\end{cases}
\label{eq:h1tohtransfer}
\end{equation}
with the running convention that
$$
h\;=\;\begin{pmatrix}h_1&0\\0&h_2\end{pmatrix}.
$$
All quasi-character of $H(\Adeles)$ we consider will be of this form. We apply the same formalism locally and also in the rational setting of $(\lieh,L)$-modules.

We consider $\eta$ as a character of $H(\Adeles)$ via the rule
$$
(1\otimes\eta)(h)\;:=\;\eta(\det(h_2)).
$$

We remark that if $\Pi$ is cohomological with associated weight $w_1$ as in \eqref{eq:prearithmeticitiyweights}, then (cf.\ Lemma 3.6.1 in \cite{grobnerraghuram2014}),
\begin{equation}
\eta_\infty\;=\;\left(\mathcal N^{\otimes{w_1}}\right)_\CC.
\label{eq:etainfinity}
\end{equation}
This implies that $\eta$ is an algebraic Hecke character and therefore $\QQ(\eta)$ is a number field. Furthermore by Wee Teck Gan's Theorem in the Appendix of \cite{grobnerraghuram2014}, if $\Pi$ admits an $(\eta,\psi)$-Shalika model, then for each $\sigma\in\Aut(\CC/\QQ)$ the twisted representation $\Pi^\sigma$ admits an $(\eta^\sigma,\psi)$-Shalika model.

\subsection{The archimedean zeta integral}\label{sec:shalikaarchimedeanintegral}

For each quasi-character $\chi\in\Hom(H_1(\RR),\CC^\times)$ and each $w\in\mathscr{S}(\Pi_\infty,\eta_\infty,\psi_\infty)$, we consider the archimedean integral
\begin{equation}
\Psi_\infty:\quad (\chi,w)\;\mapsto\;\int_{H_1(\RR)}w(h_\infty)\chi(h_\infty)dh_\infty.
\label{eq:shalikaarchimedeanintegral}
\end{equation}
This integral converges absolutely, provided $\chi$ lies in a suitable right half plane. In such a half plane $\Psi_\infty(-,w)$ defines a holomorphic function in $\chi$. The map
\begin{equation}
e_\infty:\quad (\chi,w)\;\mapsto\;\frac{\Psi_\infty(\chi,w)}{L(\frac{1}{2},\Pi_\infty\otimes\chi)}
\label{eq:shalikaarchimedeanintegralratio}
\end{equation}
is well defined for {\em all} $\chi$, as the right hand side defines an entire function in $\chi$ on every connected component of $\Hom(H_1(\RR),\CC^\times)$.

Now by Theorem 1.6 of \cite{sunpreprint2}, we know that for each $\chi$ we may find a $w\in\mathscr S(\Pi_\infty,\eta_\infty,\psi_\infty)^{(K)}$ with the property that
\begin{equation}
e_\infty(\chi',w)\;\neq\; 0.
\label{eq:shalikaarchimedeanintegralnonzero}
\end{equation}
is satisfied for all $\chi'$ in the connected component containing $\chi$. By Proposition 3.1 in \cite{friedbergjacquet1993} we know that for $K$-finite $w$ the local zeta integral $\Psi_\infty(\chi,w)$ computes $L(\frac{1}{2},\Pi_\infty\otimes\chi)$ up to a holomorphic factor. In forthcoming work by Binyong Sun and the author, we show that this factor is indeed a polynomial. This in particular implies that $e_\infty(-,w)$ is locally constant in the variable $\chi$ whenever \eqref{eq:shalikaarchimedeanintegralnonzero} is satisfied for {\em all} quasi-characters $\chi$ under consideration.

We conclude that for such a $w$ the archimedean integral \eqref{eq:shalikaarchimedeanintegral} computes the $\Gamma$-factor of the standard $L$-function $L(s,\Pi\otimes\chi)$ up to a complex unit, only depending on the connected component containing $\chi$. We will see in Theorem \ref{thm:rationaltestvector} that we may indeed choose a vector $w$ which satisfies \eqref{eq:shalikaarchimedeanintegralnonzero} for all quasi-characters $\chi$.

The integral \eqref{eq:shalikaarchimedeanintegral} and the ratio \eqref{eq:shalikaarchimedeanintegralratio} defines for $\chi$ in a suitable right half plane resp.\ for general $\chi$ defines, as a function in the variable $w$, by \eqref{eq:h1tohtransfer} and the property \eqref{eq:shalikasymmetry}, a continuous $H(\RR)$-equivariant functional
$$
e_\infty(\chi,-):\quad\mathscr S(\Pi_\infty,\eta_\infty,\psi_\infty)\to(\chi^{-1}\otimes(1\otimes\eta_\infty))_{\CC}.
$$
By \eqref{eq:etainfinity} the field of definition of the quasi-character $\chi^{-1}\otimes(1\otimes\eta)$ agrees with the field of definition of $\chi$. If we set
$$
\chi_{1}\;:=\;\chi[w_1],
$$
then we may interpret $e_\infty(\chi,-)$ as an $H(\RR)$-equivariant functional
$$
e_\infty(\chi,-):\quad\mathscr S(\Pi_\infty,\eta_\infty,\psi_\infty)\to\chi^{-1}_{1,\CC}.
$$

Again we remark that $\Pi_\infty\otimes\chi\cong\Pi_\infty$ for every finite order character $\chi$, implies the identity of local $L$-functions
$$
L(s,\Pi_\infty\otimes\chi)\;=\;L(s,\Pi_\infty).
$$

\subsection{The non-archimedean zeta integrals}

Fix a rational prime $p$. For each quasi-character $\chi\in\Hom(H_1(\QQ_p),\CC^\times)$ and each $w\in\mathscr{S}(\Pi_p,\eta_p,\psi_p)$, we consider the non-archimedean analogue of \eqref{eq:shalikaarchimedeanintegral}, given by
\begin{equation}
\Psi_p:\quad (\chi,w)\;\mapsto\;\int_{H(\QQ_p)}w(h_p)\chi(h_p)dh_p.
\label{eq:shalikanonarchimedeanintegral}
\end{equation}
For $\chi$ in a suitable right half plane the integral in \eqref{eq:shalikanonarchimedeanintegral} is absolutely convergent and $\Psi_p(-,w)$ defines a holomorphic function in $\chi$. The map
\begin{equation}
e_p:\quad (\chi,w)\;\mapsto\;\frac{\Psi_p(\chi,w)}{L(\frac{1}{2},\Pi_p\otimes\chi)}
\label{eq:shalikanonarchimedeanintegralratio}
\end{equation}
is well defined for all $\chi$ and $w$ and defines an entire function in the variable $\chi$ on each connected component of $\Hom(H_1(\QQ_p),\CC^\times)$. As in the archimedean case we always find for a given $\chi$ a $w\in\mathscr S(\Pi_p,,\eta_p,\psi_p)$ such that
$$
e_p(\chi,w)\;\neq\;0.
$$
This implies that for each quasi-character $\chi$ there is a {\em good vector} $t_p^\chi$, depending only on the connected component of $\chi$ in $\Hom(H_1(\QQ_p),\CC^\times)$, satisfying
$$
e_p(\chi,t_p^\chi)\;=\;1,
$$
cf.\ Proposition 3.1 in \cite{friedbergjacquet1993}. For almost all primes
$$
t_p^0\in\mathscr S(\Pi_p,\eta_p,\psi_p)^{G(\ZZ_p)}
$$
produces the correct Euler factor for $\Pi_p\otimes\chi_p$.

\subsection{The global zeta integral}

We depart from a quasi-character
$$
\chi:\quad H_1(\QQ)\backslash H_1(\Adeles)\to\CC^\times,
$$
and fix a good test vector
$$
t^\chi\;=\;\otimes_v t_v^{\chi_v}\;\in\;\otimes_v'\mathscr S(\Pi_v,\eta_v,\psi_v)
\;=\;\mathscr S(\Pi,\eta,\psi),
$$
which at almost all places $v$ agrees with the normalized spherical vector $t_v^0\in\mathscr S(\Pi_v,\eta_v,\psi_v)$. Then the cusp form
$$
\Theta(t^\chi)\;:=\;\left(\mathcal S_{\psi}^\eta\right)^{-1}(t^\chi)\;\in\;\Pi,
$$
is a good test vector for the global integral
$$
\Lambda(s,\Pi\otimes\chi)\;=\;I(s,\chi,\Theta(t^\chi))\;:=\;
$$
$$
\int_{Z(\Adeles)H(\QQ)\backslash H(\Adeles)}
\Theta(t^\chi)\left(
h
\right)
\cdot
\eta(\det(h_2^{-1}))
\cdot
(\chi\otimes\omega_{s-\frac{1}{2}}^0)\left(
h_1h_2^{-1}\right)
d
h.
$$
Again we write $L(s,\Pi\otimes\chi)$ for the corresponding incomplete $L$-function.

\section{Rationality properties of archimedean Shalika integrals}\label{sec:cohomologicalinduction}

In this section we discuss how to apply the rational theory of the first sections to the case at hand.

\subsection{Cohomologically induced standard modules}

We fix a $\theta$-stable Borel subalgebra $\lieq_{E_K}\subseteq\lieg_{E_K}$ transversal to $\lieh$, i.e.
\begin{equation}
\theta(\lieq_{E_K})\;=\;\lieq_{E_K},
\label{eq:thetastability}
\end{equation}
and
\begin{equation}
\lieg_{E_K}\;=\;\lieq_{E_K}\oplus\lieh_{E_K}.
\label{eq:transversality}
\end{equation}
That a Borel subalgebra $\lieq$ with these properties exists over $\CC$ is shown in Section 4 of \cite{sunpreprint2}, and it is easy to deduce from the discussion there that $\lieq$ is defined over $E_K$.

By \eqref{eq:thetastability}, the Lie algebra
$$
\lieq_{E_K}\cap\liek_{E_K}\subseteq\liek_{E_K}
$$
is a Borel subalgebra of $\liek_{E_K}$. Again by \eqref{eq:thetastability}, complex conjugation $\tau_K\in\Gal(E_K/\QQ)$ maps $\lieq$ to its opposite
$$
\overline{\lieq}_{E_K}\;:=\lieq_{E_K}^{\tau_K}\;=\;\lieq_{E_K}^-,
$$
if we consider the unique $\theta$-stable and $\QQ$-rational Levi factor given explicitly by
$$
\liec\;:=\;\lieq\cap\overline{\lieq}.
$$
We write $\lieu\subseteq\lieq$ for the nilpotent radical of $\lieq$. Then $\lieq_{E_K}\cap\liek_{E_K}$ has the Levi decomposition
$$
\lieq_{E_K}\cap\liek_{E_K}\;=\;(\liec_{E_K}\cap\liek_{E_K})\oplus (\lieu_{E_K}\cap\liek_{E_K}).
$$

We fix the compact component $C$ of the Levi pair associated to $\lieq$ as the normalizer of $\liec=\lieq\cap\lieq^-$ inside $K$. We remark that $C$ is defined over $\QQ$, and $(\lieq,C)_{E_K}$ is a parabolic pair over $E_K$. Furthermore $CK^0$ is of index $2^{r_F}$ in $K$.

We consider the $(\lieg_\CC,K_\CC)$-module underlying $\Pi_\infty$ as a cohomologically induced standard module $A_{\lieq}(\mu)_\CC$ as follows.

Starting from $M(\mu)$ as a rational irreducible representation of $G$ with highest weight $\mu$, we consider $\mu$ as a highest weight with respect to $\lieq_\CC$. As such we have a one-dimensional highest weight space $H^0(\lieu;M(\mu))$ inside of $M(\mu)$. As $C$ normalizes $\lieq$, the inclusion $C\to G$ induces an action of $C$ on $M(\mu)$ under which $H^0(\lieu;M(\mu))$ is stable. Hence we obtain a character $\mu$ of the pair $(\liec,C)_\CC$, or equivalently of the parablic pair $(\lieq,C)_\CC$ with trivial action of the radical. This agrees with the natural action of the pair $(\lieq,C)_\CC$ on $H^0(\lieu;M(\mu))$. All these data are indeed defined over $E_K(\mu)=E_K\cdot\QQ(\mu)$, if we fix once and for all an inclusion
\begin{equation}
\iota:\quad E_K(\mu)\to\CC.
\label{eq:complexembedding}
\end{equation}

Recall from section \ref{sec:inducedmaps} the middle degree
$$
S_\lieq\;=\;\dim_{E_K}(\lieu_{E_K}\cap\liek_{E_K}),
$$
and consider the inclusions of pairs $i_\lieq:(\lieq,C)_{E_K}\to(\lieg,C)_{E_K}$ and $i_\lieg:(\lieg,C)_\QQ\to(\lieg,K)_\QQ$. We also consider the inclusion $i_\lieg^\circ:(\lieg,C)_\QQ\to(\lieg,CK^0)_\QQ$. Then for the associated Bernstein functors we have the relation
$$
P_{i_{\lieg}}\;=\;\ind_{CK^0}^K\circ P_{i_{\lieg}^\circ},
$$
which, by the associated degenerating Grothendieck spectral sequence, naturally extends to an isomorphism
\begin{equation}
L_qP_{i_{\lieg}}\;=\;\ind_{CK^0}^K\circ L_qP_{i_{\lieg}^\circ},
\label{eq:gammacomposition}
\end{equation}
of derived functors in any degree $q$.

We have the $E_K(\mu)$-rational modules
$$
A_\lieq(\mu)\;:=\;L_{S_\lieq}P_{i_{\lieg}\circ i_{\overline{\lieq}}}(\mu\otimes\bigwedge^{\dim\lieu}\lieu_{E_K(\mu)}),
$$
and
$$
A_\lieq^\circ(\mu)\;:=\;L_{S_\lieq}P_{i_{\lieg}^\circ\circ i_{\overline{\lieq}}}(\mu\otimes\bigwedge^{\dim\lieu}\lieu_{E_K(\mu)}).
$$
Then $A_\lieq^0(\mu)$ is a submodule of $A_\lieq(\mu)$ and the corresponding inclusion induces by \eqref{eq:gammacomposition} an isomorphism
$$
A_\lieq(\mu)\;=\;\ind_{CK^0}^K A_\lieq^\circ(\mu).
$$
In Theorem 8.26 of \cite{januszewskipreprint} it is shown that $A_\lieq(\mu)$ is defined over $\QQ(\mu)$. The field of definition of $A_\lieq^\circ(\mu)$ is determined in Theorem 2.3 of \cite{januszewskipart1}.

Now if $\Pi_\infty$ has non-trivial relative Lie algebra cohomology with coefficients in $M(\mu)^\vee$, then we know by \cite{voganzuckerman1984} and the fact that Bernstein functors commute with base change (Theorem \ref{thm:basechangeinduction}), that we have an isomorphism
\begin{equation}
\iota_\infty:\quad A_\lieq(\mu)_\CC\;\to\;\Pi_\infty^{(K)}
\label{eq:moduleisomorphism}
\end{equation}
where the base change on the left hand side is understood via the embedding \eqref{eq:complexembedding}.

We remark that the natural embedding $\QQ(\mu)\to E_K(\mu)$ induces with \eqref{eq:complexembedding} an embedding $\QQ(\mu)\to\CC$, which we also consider fixed in the sequel. In particular \eqref{eq:moduleisomorphism} retains its meaning also for the model of $A_{\lieq}(\mu)$ over $\QQ(\mu)$, the latter being unique by Proposition 3.5 of \cite{januszewskipreprint}.

\subsection{Structure of the bottom layer}

We denote by
$$
B_{\lieq}(\mu)_{E_K(\mu)}\;:=\;
A_{\lieq\cap\liek}(\mu|_{C}\otimes\!\!\!\!\!\bigwedge^{\dim\lieu/\lieu\cap\liek}\!\!\!\!\!\!\!(\lieu/\lieu\cap\liek))_{E_K(\mu)}
\;\subseteq\;
A_{\lieq}(\mu)_{E_K(\mu)}
$$
the bottom layer, with
$$
A_{\lieq\cap\liek}(\cdot)\;
\;:=\;
L_{S_\lieq}P_{i_{\liek}i_{\overline{\lieq}\cap\liek}}\left((\cdot)\otimes\!\!\!\!\!\bigwedge^{\dim\lieu\cap\liek}\!\!\!\!\!\lieu\cap\liek\right).
$$
The main descent argument in the proof of Theorem 7.3 in \cite{januszewskipreprint} is based on the fact that the bottom layer itself is defined over $\QQ(\mu)$. Again its model over $\QQ(\mu)$ is unique and we denote it the same.

As in the non-compact case we introduce the intermediate representation
$$
B_{\lieq}^\circ(\mu|_{C})_{E_K(\mu)}:=A_{\lieq\cap\liek}^\circ(\mu|_{C})_{E_K(\mu)}
:=
L_{S_\lieq}P_{i_{\liek}^\circ i_{\overline{\lieq}\cap\liek}}
(\mu|_{C}\otimes\bigwedge^{\dim\lieu}\lieu_{E_K(\mu)})
\subseteq A_{\lieq}^\circ(\mu)_{E_K(\mu)},
$$
with respect to the inclusion
$$
i_{\liek}^\circ:\quad (\liek,C)\to(\liek,CK^0).
$$
Then $B_{\lieq}^\circ(\mu|_{C})_{E_K(\mu)}$ is an irreducible representation of $K^0$ of highest weight
\begin{equation}
\mu_K\;:=\;\mu|_{C^0}+2\rho(\lieu/\lieu\cap\liek),
\label{eq:ktypeweight}
\end{equation}
where
$$
\rho(\lieu/\lieu\cap\liek)\;=\;\frac{1}{2}\sum_{\alpha\in\Delta(\lieu/\lieu\cap\liek,\liec)}\alpha,
$$
is the half sum of weights in $\lieu/\lieu\cap\liek$. Again we have
\begin{equation}
B_\lieq(\mu)_{E_K(\mu)}\;=\;\ind_{CK^0}^{K}B_\lieq^\circ(\mu)_{E_K(\mu)}.
\label{eq:bqintermediateinduction}
\end{equation}

By the preceeding discussion and the transitivity principle \eqref{eq:bqintermediateinduction}, the representation $B_{\lieq}(\mu)_{\CC}^\circ$ decomposes a complex representation of
$$
K^0(\RR)\;=\;\prod_{v}\SO(2n,F_v),
$$
over $\CC$, into an outer tensor product
\begin{equation}
B_{\lieq}^\circ(\mu)_{\CC}\;=\;
\bigotimes_{v}B_{\mu_{v}},
\label{eq:bottomlayerzerodecomposition}
\end{equation}
where
$$
\mu_K\;=\;\otimes_v \mu_{v},
$$
is the natural factorization of $\mu_K$ into local components and $B_{\mu_{v}}$ denotes the complex irreducible representation of $\SO(2n,F_v)$ of highest weight $\mu_{v}$.

Then, again as a representation of $K^0(\RR)$, the full bottom layer decomposes as
\begin{equation}
B_{\lieq}(\mu)_{\CC}\;=\;
\bigotimes_{v}\left(B_{\mu_{v}}\oplus B_{\tilde{\mu}_{v}}\right),
\label{eq:bottomlayerdecomposition}
\end{equation}
with weights $\tilde{\mu}_v$, depending on $\mu_v$. 

We remark that as the representation $B_{\lieq}^\circ(\mu)_{\CC}$ is defined over $E_K(\mu)$, the summands in the direct sum decomposition \eqref{eq:bottomlayerdecomposition} are defined over $E_K(\mu)$ as well. However for later use, we need precise knowledge about their field of definition.

\begin{theorem}[Theorem 2.3 in \cite{januszewskipart1}]\label{thm:connectedbottomlayerdefinition}
The $K^0$-module $B_{\lieq}^\circ(\lambda)_{\CC}$ resp.\ the $(\lieg,K^0)$-module $A_{\lieq}^\circ(\lambda)_\CC$ is defined over $E_K(\mu)$. If $\sqrt{-1}\not\in\QQ(\mu)$, then $B_{\lieq}^\circ(\lambda)_{\CC}$ resp.\ $A_{\lieq}^\circ(\lambda)_\CC$ is defined over $\QQ(\mu)$ if and only $2\mid n$. The same applies to the summands in the direct sum decomposition \eqref{eq:bottomlayerdecomposition}.
\end{theorem}

For $2\nmid n$, the decomposition \eqref{eq:bottomlayerdecomposition} corresponds to the decomposition of $A_{\lieq}(\mu)$ into tensor products of holomorphic resp.\ antiholomorphic discrete series representations. In this case each weight $\tilde{\mu}_v$ is the highest weight of the {\em dual} of $B_{\mu_v}$. Otherwise this statement is not true.

For any archimedean place $v$ of $F$ we choose an $l_v\in K_n(\QQ)$ with the property that its image in
$$
\pi_0(K_n)\;\cong\;\{\pm1\}^{r_F}
$$
is non-trivial in the factor $\{\pm1\}$ corresponding to $v$ and trivial in the factors corresponding to other real places. We assume without loss of generality that $l_v$ is chosen in the image of a homomorphic section
$$
\pi_0(K_n)\to K_n.
$$
We consider $l_v$ as an element of $L(\QQ)$ via the respective embeddings
\begin{equation}
G_n\;=\;H_1\cap L\;\subseteq\;L.
\label{eq:gnembedding}
\end{equation}
Then the image of $l_v$ in $\pi_0(K)$ corresponds likewise under the isomorphism
$$
\pi_0(K)\;\cong\;
\{\pm1\}^{r_F},
$$
to the element whose $v$-component is $-1$, and trivial otherwise. Therefore the action of $l_v$ on $B_{\lieq}(\mu)$ interchanges the two modules $B_{\mu_v}$ and $B_{\mu_v}^\vee$ and leaves the other factors $B_{\mu_{v'}}$, $v'\neq v$, in \eqref{eq:bottomlayerzerodecomposition} invariant. In particular we see that any irreducible $K^0$-submodule of $B_{\lieq}(\mu)_\CC$ generates $B_{\lieq}(\mu)_\CC$ as a representation of $L$ resp.\ $L\cap H_1$.

\begin{proposition}\label{prop:bonedimensionalfunctionals}
For each character $\chi\in X_\CC(L)$ we have
$$
\dim_\CC\Hom_L(B_{\lieq}(\mu)_\CC,\chi_\CC)\;\leq\;1.
$$
\end{proposition}

\begin{proof}
The elementary proof proceeds as in the proof of Proposition 4.1 in \cite{januszewskipart1}, as all is needed is the transversality condition \eqref{eq:transversality}.
\end{proof}

\begin{proposition}\label{prop:finiteness}
For $n\geq 1$ there exists a vector
$$
t_0\;\in\;B_{\lieq}(\mu)_{\QQ(\mu)}
$$
with the property that for any character $\chi\in X_{\CC}(L)$ and any
$$
0\neq\lambda\;\in\;\Hom_{L^0}(B_{\lieq}(\mu)_\CC,\chi_\CC)
$$
we have
$$
\lambda(t_0)\;\neq\; 0.
$$
\end{proposition}

\begin{proof}
The proof again proceeds as the proof of Proposition 4.2 in \cite{januszewskipart1}.
\end{proof}

\subsection{Rationality of functionals}

The following result is at the heart of our period relations.

\begin{theorem}\label{thm:rationalfunctionals}
For each critical quasi-character $\chi$ of $H_1(\RR)$ the archimedean functional
$$
e_\infty(\chi,-):\quad \mathscr S(\Pi_\infty,\eta_\infty,\psi_\infty)^{(K)}\to\chi_{1,\CC}^{-1}
$$
is defined over $\QQ(\mu)$.
\end{theorem}

The theorem remains valid more generally for any quasi-character which is a norm translate of a finite order character.

\begin{proof}
We first recall the isomorphism $\iota_\infty$ from \eqref{eq:moduleisomorphism}. By the argument in the proof of Theorem 1.7 in section 4 in \cite{sunpreprint2}, we know that $e_\infty(\chi,-)$ agrees, up to scalar multiples and composition with $\iota_\infty$, with the cohomologically induced functional
$$
\eta_{\lieh,0}:A_{\lieq}(\mu)_\CC\to\chi_{1,\CC}^{-1},
$$
induced from any non-zero map
$$
\eta_{\lieq\cap\liep}:\quad \mathcal F_{k_{C}}(\mu\otimes\bigwedge^{\dim\lieu}\lieu)\to\mathcal F_{k_{L}}^{\vee}(\chi_{1}^{-1}).
$$
Since $\chi_{1}$ is defined over $\QQ$, and the left hand side is defined over $E_K(\mu)$ we may and do assume that $\eta_{\lieq\cap\liep}$ is defined over $E_K(\mu)$.

The auxiliary data for the construction of the functional as in section \ref{sec:inducedmaps} is given by $\liep=\lieh$, $D=L$, $S_\liep=q=0$, together with the tautological inclusion maps. Then $\lieq\cap\lieh=0$ by transverality.

If $\sqrt{-1}\in \QQ(\mu)$ we are already done. Otherwise we need to study the effect of complex conjugation, as an automorphism $\tau$ of the extension $E_K(\mu)/\QQ(\mu)$, on $\eta_{\lieh,0}$. To simplify notation we set $E:=E_K(\mu)$ and $E^+:=\QQ(\mu)$.

We remark that complex conjugation stabilizes the domain an codomain of $\eta_{\lieh,0}$. It also interchanges $\lieq$ and $\lieq^-=\overline{\lieq}$, and sends the $\lieu$-highest weight $\mu$ to the $\overline{\lieu}$-highest weight
$$
\overline{\mu}\;\subseteq\;H^0(\overline{\lieu};M(\mu)_{E_K}).
$$

On the one hand, since $A_\lieq(\mu)$ is defined over $E^+$, we have the commutative square
\begin{equation}
\begin{CD}
A_\lieq(\mu)_{E^+}\otimes_{E^+} E@>i>> A_\lieq(\mu)_{E}\\
@V{{\bf1}\otimes\tau}VV @VV{\tau}V\\
A_\lieq(\mu)_{E^+}\otimes_{E^+} E@>>\iota>A_{\overline{\lieq}}(\overline{\mu})_{E}
\end{CD}
\label{eq:Aqtauiso}
\end{equation}
On the other hand, by Theorem \ref{thm:etabasechange}, we have the commutative square
\begin{equation}
\begin{CD}
A_\lieq(\mu)_{E}@>\eta_{\liel,0}>>\chi_{1,E}^{-1}\\
@V{{\bf1}\otimes\tau}VV @VV{\tau}V\\
A_{\overline{\lieq}}(\overline{\mu})_{E}@>\eta_{\liel^\tau,0}>>\chi_{1,E}^{-1}
\end{CD}
\label{eq:Aqtausquare}
\end{equation}
It remains to show that for a suitable choice of $\eta_{\lieq\cap\liep}$, the two functionals $\eta_{\liel,0}$ and $\eta_{\liel^\tau,0}\circ\iota$ on $A_\lieq(\mu)_{E^+}\otimes_{E^+} E$ agree, since we already know by \eqref{eq:Aqtauiso} and \eqref{eq:Aqtausquare} that the diagram
$$
\begin{CD}
A_\lieq(\mu)_{E}@>\eta_{\lieh,0}>>\chi_{1,E}^{-1}\\
@V{\tau}VV @VV{\tau}V\\
A_{\lieq}(\mu)_{E}@>\eta_{\lieh^\tau,0}\circ\iota>>\chi_{1,E}^{-1}
\end{CD}
$$
commutes.

From the description of the maps $\eta_{\liel,0}$ and $\eta_{\liel^\tau,0}$ on the level of complexes given in Corollary \ref{cor:expliciteta} we deduce that, once $\pi^*$ is chosen over $E^+$, there is indeed a choice of $\eta_{\lieq\cap\liep}$ as claimed. Since $\iota_\infty$ is defined over $E^+=\QQ(\mu)$, this proves the statement of the theorem.
\end{proof}

\section{Proof of the period relations}

In this section we proof the expected period relations. We proceed as in section 4 of \cite{januszewskipart1}.

We will make implicit use of the embedding \eqref{eq:gnembedding}. As before we consider quasi-characters of $H_1$ also as quasi-characters of $H$ via \eqref{eq:h1tohtransfer}. We apply the same convention to algebraic and critical quasi-characters, and implicitly assume that they arise from $H_1$, and also to one-dimensional $(\lieh_1,L_1)$-modules, as the notation suggests.

Finally we recall that weintroduced for a quasi-character $\chi$ the quasi-character $\chi_1=\chi[w_1]$. As above we simply write $\chi_1$ to denote this character.

\subsection{Rational test vectors}

Writing $\QQ(\mu,\chi)$ for the composite of $\QQ(\mu)$ and the field of definition $\QQ(\chi)$ of the quasi-character $\chi$, as a fixed subfield of $\CC$, we obtain
\begin{proposition}[{Proposition 1.1 in \cite{januszewskipreprint}}]\label{prop:functionalrationality}
For any quasi-character $\chi$ of $H_1(\RR)$ we have
$$
\Hom_{\lieh,L}(A_{\lieq}(\mu)_\CC,\chi_{1,\CC})\;=\;
\Hom_{\lieh,L}(A_{\lieq}(\mu)_{\QQ(\mu,\chi)},\chi_{1,\QQ(\mu,\chi)})\otimes_{\QQ(\mu,\chi)}\CC.
$$
\end{proposition}

Each one-dimensional $(\lieh_1,L_1)_\CC$-module $\chi_\CC$ corresponds bijectively to a quasi-character $\chi$ of $H_1(\RR)$ and by composing the functional $e_\infty(\chi,\cdot)$ of section \ref{sec:shalikaarchimedeanintegral} with the fixed isomorphism $\iota_\infty$ in \eqref{eq:moduleisomorphism} we obtain a non-zero $(\lieh,L)_\CC$-equivariant functional
\begin{equation}
\lambda_{\chi,\CC}:= e_\infty(\chi,\cdot)\circ\iota_\infty:\quad A_\lieq(\mu)_\CC\to \chi^{-1}_{1,\CC}.
\label{eq:complexfunctional}
\end{equation}

\begin{theorem}\label{thm:rationaltestvector}
There exists a rational test vector $t_0\in B_\lieq(\mu)_{\QQ(\mu)}$ with the property that for every $\chi\in\Hom(H_1(\RR),\CC^\times)$,
\begin{equation}
\lambda_{\chi,\CC}(t_0)\;\neq\; 0.
\label{eq:testvectorcondition}
\end{equation}
In particular for every quasi-character $\chi$ there exists a constant $c_\chi\in\CC^\times$, only depending on the connected component containing $\chi$ in $\Hom(H_1(\RR),\CC^\times)$, with the property that
$$
e_\infty(\chi,\iota_\infty(t_0))\;=\;c_\chi.
$$
\end{theorem}

\begin{proof}
By \cite{sunpreprint2}, or alternatively by Theorem \ref{thm:nonvanishing}, we know that
\begin{equation}
0\neq\lambda_{\chi,\CC}|_{B_{\lieq}(\mu)_\CC}\in\Hom_{L}(B_{\lieq}(\mu)_\CC,\chi_1|_{L.\CC}^{-1}).
\label{eq:resnonzero}
\end{equation}
Therefore any choice of $t_0$ as in Proposition \ref{prop:finiteness} satisfies \eqref{eq:testvectorcondition}. By the archimedean theory in the Shalika case discussed in section \ref{sec:shalikaarchimedeanintegral}, the condition
$$
\forall\chi:\quad e_\infty(\chi,\iota_\infty(t_0))\;\neq\;0
$$
implies the claim.
\end{proof}

\begin{corollary}
For any $t\in A_{\lieq}(\mu)_{\QQ(\mu)}$ and any algebraic quasi-character $\chi\in X^{\rm alg}(H_1(\RR))$ we have
$$
e_\infty(\chi,\iota_\infty(t))\;\in\;\QQ(\mu)\cdot c_\chi.
$$
\end{corollary}

\begin{proof}
We know by Theorem \ref{thm:rationalfunctionals} that the functional \eqref{eq:complexfunctional} is defined over $\QQ(\mu)$. Therefore the image of the rational subspace
$$
A_{\lieq}(\mu)_{\QQ(\mu)}\;\subseteq\; A_{\lieq}(\mu)_{\CC}
$$
under $\lambda_{\chi,\CC}$ is a one-dimensional $\QQ(\mu)$-subspace of $\chi_\CC^{-1}\cong\CC$. This subspace contains $\lambda_{\chi,\CC}(t)$ and $\lambda_{\chi,\CC}(t_0)$. Therefore the claim follows from Theorem \ref{thm:rationaltestvector}.
\end{proof}

\subsection{The archimedean period relation}

The analogue of Theorem 4.10 in \cite{januszewskipart1} in the Shalika case is

\begin{theorem}\label{thm:icontribution}
For any admissible algebraic quasi-character $\chi$ of $H(\RR)$ we have
$$
c_{\sgn_\infty\otimes\chi}\;\in\;\QQ(\mu)\cdot (i^{r_Fn}\cdot c_\chi).
$$
\end{theorem}

\begin{proof}
The proof proceeds as the proof of Theorem 4.10 in loc.\ cit. We sketch the argument for the convenience of the reader.
According to the direct sum decomposition \eqref{eq:bottomlayerdecomposition}, which is already defined over $E_K(\mu)$, we may write
\begin{equation}
t\;=\;
\sum_{\varepsilon\in\pi_0(H_1)}t_\varepsilon
\label{eq:tsumrepresentation}
\end{equation}
where
$$
t_\varepsilon\;\in\;\varepsilon\cdot B_\lieq^\circ(\mu)_{E_K(\mu)}\;\subseteq\;B_\lieq(\mu)_{E_K(\mu)},
$$
for our choice of representatives
$$
\varepsilon\;=\;\prod_{v}l_v^{\delta_v}\;\in\; L(\QQ),\quad \delta_v\in\{0,1\}.
$$
They satisfy
$$
G(\RR)\;=\;\bigsqcup_{\varepsilon} G(\RR)^0C(\RR)\cdot\varepsilon,
$$
where the contribution of $C(\RR)$ is trivial.

Then the archimedean Shalika integral \eqref{eq:shalikaarchimedeanintegral} decomposes into the sum of the integrals
$$
\Psi_\infty^\varepsilon:
\quad (\chi,w)\;\mapsto\;
\int_{H_1^0(\RR)}w(h_\infty)\chi(h_\infty)dh_\infty.
$$
Similarly we have the entire functions of $\chi$,
$$
e_\infty^\varepsilon:\quad (\chi,w)\;\mapsto\;\frac{\Psi_\infty^\varepsilon(\chi,w)}{L(\frac{1}{2},\Pi_\infty\otimes\chi)},
$$
satisfying
$$
e_\infty(\chi,\iota_\infty(t))\;=\;\sum_{\varepsilon} e_\infty^{\varepsilon}(\chi,\iota_\infty(t_{\varepsilon})).
$$
As in the proof of Theorem 4.10 in \cite{januszewskipart1} we conclude that
\begin{equation}
\lambda_{\chi,E_K(\mu)}(t)\;=\;
\sum_{\varepsilon}\lambda_{\chi,E_K(\mu)}(t_\varepsilon)\;=\;
\sum_{\varepsilon}\sgn_v(\varepsilon)\cdot\lambda_{\chi\otimes\sgn_v,E_K(\mu)}(t_\varepsilon).
\label{eq:integralcomparison}
\end{equation}

Complex conjugation $\tau\in\Gal(E/E^+)$ permutes the direct factors in \eqref{eq:bottomlayerdecomposition}. By Theorem \ref{thm:connectedbottomlayerdefinition} this action is trivial if and only if $2\mid m$.

If $\sqrt{-1}\in\QQ(\mu)$ or $2\mid m$, then the vectors $t$ and $t_{\pm v_0}$ all lie in the same $\QQ(\mu)$-rational model $B(\mu)_{\QQ(\mu)}$, and the claim follows by the rationality of the archimedean functional.

Now suppose $\sqrt{-1}\not\in\QQ(\mu)$ and $2\nmid m$. For any real place $v_0$ of $F$ we consider the $K^0$-submodule
$$
B_{v_0,E_K(\mu)}\;:=\;\sum_{\varepsilon\in\kernel\sgn_{v_0}}\varepsilon\cdot B_{\lieq}^\circ(\mu)_{E_K(\mu)}\;\subseteq\;B_{\lieq}(\mu)_{E_K(\mu)},
$$
where the sum ranges over all possible products $\varepsilon$ of the elements $l_v$ with $v\neq v_0$. Then
\begin{equation}
B_{\lieq}(\mu)_{E_K(\mu)}\;=\;B_{v_0,E_K(\mu)}\oplus l_{v_0}\cdot B_{v_0,E_K(\mu)}.
\label{eq:bottomlayervdecomposition}
\end{equation}
Since $2\nmid m$, the second direct summand on the right hand side is naturally identified with the dual of $B_{v_0,E_K(\mu)}$. We conclude that $\tau$ interchanges the two direct summands in \eqref{eq:bottomlayervdecomposition}.

In particular $\tau$ sends the vector
$$
t_{v_0}\;:=\;\sum_{\varepsilon\in\kernel\sgn_{v_0}} t_\varepsilon
$$
to
$$
t_{v_0}^{\tau}\;\in\;l_{v_0}\cdot B_{v_0,E_K(\mu)}.
$$
Now $t$ is $\QQ(\mu)$-rational, and thus invariant under $\tau$. By the uniqueness of the decomposition \eqref{eq:tsumrepresentation} we see that for each $\varepsilon$,
$$
t_{v_0}^{\tau}\;=\;\sum_{\varepsilon\in\kernel\sgn_{v_0}} t_{l_{v_0}\varepsilon}\;=:\;t_{-v_0}.
$$
Hence the vector
$$
t_{v_0}-t_{-v_0}\;\in\;\sqrt{-1}\cdot B_\lieq(\mu)_{\QQ(\mu)}\;\subseteq\;B_\lieq(\mu)_{E_K(\mu)}.
$$
is \lq{}purely imaginary\rq{}, and consequently
\begin{equation}
\sqrt{-1}\cdot(t_{v_0}-t_{-v_0})\;\in\; B_\lieq(\mu)_{\QQ(\mu)}\;\subseteq\;B_\lieq(\mu)_{E_K(\mu)}.
\label{eq:tauactionoontepsilon}
\end{equation}
Again as in loc.\ cit.\ we obtain
$$
\lambda_{\chi,E_K(\mu)}(t) \;=\;\lambda_{\chi\otimes\sgn_v,\QQ(\mu)}(t)\cdot \sqrt{-1}\cdot\QQ(\mu),
$$
where the last relation follows from the rationality property of the functional proven in Theorem \ref{thm:rationalfunctionals}. This proves the claim in the case $\sqrt{-1}\not\in\QQ(\mu)$ and $2\nmid m$.
\end{proof}

\begin{corollary}\label{cor:icontribution}
For any $t\in A_{\lieq}(\mu)_{\QQ(\mu)}$, any algebraic quasi-character $\chi$ of $H_1(\RR)$ and any $k\in\ZZ$ we have
$$
\lambda_{\chi[k],\QQ(\mu)}(t)\;\in\;\QQ(\mu)\cdot (i^{kr_Fn}\cdot c_\chi).
$$
\end{corollary}

\begin{proof}
Follows from Theorems \ref{thm:rationaltestvector} and \ref{thm:icontribution} in the same way as Corollary 4.11 in \cite{januszewskipart1} follows from Theorems 4.8 and 4.10 therein.
\end{proof}

We call an algebraic $\chi=\sgn^{\delta}\otimes\left(\mathcal N^{\otimes k}\right)$ is {\em critical} for $\Pi$ (or $\Pi_\infty$), if $L(s,\Pi_\infty)$ and $L(1-s,\Pi_\infty^\vee)$ both have no pole at $s=k+\frac{1}{2}$. For critical $\chi$ we know that
$$
L_\infty(\frac{1}{2},\Pi_\infty\otimes\chi)\;\neq\;0.
$$
By Theorem \ref{thm:rationaltestvector} we find a $t_0\in A_{\lieq}(\mu)_{\QQ(\mu)}$ satisfying
$$
\Psi_\infty(\chi,\iota_\infty(t_0))\;\neq\;0,
$$
for any critical $\chi$ ($\Psi_\infty(\chi,\iota_\infty(t_0))$ is well defined for critical $\chi$ outside the region of absolute convergence by holomorphic continuation). The analogue of Corollary 4.12 of loc.\ cit.\ is
\begin{corollary}\label{cor:criticalrelation}
For any $t\in A_{\lieq}(\mu)_{\QQ(\mu)}$, any pair of critical quasi-characters $\chi,\chi'$ of $H_1(\RR)$ with
$$
\chi'\;=\;\chi[k]\otimes(\sgn_\infty)^{\delta},\quad k\in\ZZ,\,\delta\in\{0,1\},
$$
we have
$$
\Psi_\infty(\chi',\iota_\infty(t))\;\in\;
\QQ(\mu)\cdot i^{(k+\delta)r_Fn}\cdot
\Psi_\infty(\chi,\iota_\infty(t_0))\cdot
\frac{L_\infty(\frac{1}{2}+k,\Pi_\infty\otimes\chi)}{L_\infty(\frac{1}{2},\Pi_\infty\otimes\chi)},
$$
for every $t\in A_{\lieq}(\mu)_{\QQ(\mu)}$.
\end{corollary}

\subsection{Relative Lie algebra cohomology over $\QQ(\mu)$}

We define the reductive group
$$
GK=\{g\in G\mid \exists z\in Z^{\rm s}:\;zg=\theta(g)\}\;\subseteq\;G.
$$
It is the product of $K$ with the maximal $\QQ$-split torus in the center $Z$ of $G$ and therefore defined over $\QQ$. Since
$$
K/K^0C\;=\;GK/GK^0C\;=\;
\pi_0(L\cap H_1)
$$
Shapiro's Lemma implies
\begin{eqnarray*}
H^{\bullet}(\lieg,GK^0; A_{\lieq}(\mu)\otimes M(\mu)^\vee)_{E_K(\mu)}
&=&
H^{\bullet}(\lieg,GK^0; A_{\lieq}^\circ(\mu)\otimes M(\mu)^\vee)_{E_K(\mu)}\otimes \QQ[\pi_0(L_1)]\\
&=:&
H^{\bullet}(\lieg,GK^0; A_{\lieq}^\circ(\mu)\otimes M(\mu)^\vee)[\pi_0(L_1)]_{E_K(\mu)},
\end{eqnarray*}
as $\pi_0(L_1)$-modules.

Introduce the \lq{}real\rq{} top degree
$$
t_{2n}^{\RR}\;:=\;
n^2+n-1,
$$
which is the lowest resp.\ highest degree for which the relative Lie algebra cohomology of non-degenerate cohomological representations of $\GL_n(\RR)$ resp.\ $\GL_{2n}(\RR)$ does not vanish, and set
$$
d\;:=\;
\sum_{v}t_{2n}^{F_v}+[F:\QQ]-1.
$$
This is the top degree of Lie algebra cohomology for non-degenerate cohomological representations of $G(\RR)$. As the cohomology of $A_{\lieq}^\circ(\mu)$ in the degree $d$ is one-dimensional, the standard descent argument \cite[Proposition 1.1]{januszewskipreprint} together with the Homological Base Change Theorem (or Theorem \ref{thm:basechangeinduction} in the present article) shows that we have a natural isomorphism of $\pi_0(L_1)$-modules
\begin{equation}
H^{d}(\lieg,GK^0; A_{\lieq}(\mu)\otimes M(\mu)^\vee)_{\QQ(\mu)}\;=\;
\QQ(\mu)[\pi_0(L_1)].
\label{eq:cohomologycomponentaction}
\end{equation}
Finally we remark that we have a natural isomorphism
\begin{equation}
H^{\bullet}(\lieg,GK^0; A_{\lieq}(\mu)\otimes M(\mu)^\vee)_{\QQ(\mu)}\;=\;
H^0(GK^0; \bigwedge^\bullet(\lieg/\liegk)^\vee\otimes 
A_{\lieq}(\mu)\otimes M(\mu)^\vee)_{\QQ(\mu)}
\label{eq:cohomologyrepresentation}
\end{equation}
of $\pi_0(L_1)$-modules, as the standard complex computing the relative Lie algebra cohomology degenerates in our case. This is well known over $\CC$ (cf.\ combine Proposition 3.1 in Borel-Wallach \cite{borelwallach2000} with Proposition 2.3 in \cite{januszewskipreprint}), and this already implies the claim over $\QQ(\mu)$ by Proposition 1.1 in loc.\ cit.

\subsection{Cohomological test vectors}

We proceed as in \cite{januszewskipart1} with minor modifications adapted to the Shalika case.

Any cohomology class in
$$
H^{d}(\lieg,GK^0; A_{\lieq}(\mu)\otimes M(\mu)^\vee)_{\QQ(\mu)}
$$
has by \eqref{eq:cohomologyrepresentation} a unique representative
$$
h\;=\;
\sum_{p=1}^s \omega_p\otimes a_p\otimes m_p
\;\in\;H^0(GK^0; \bigwedge^d(\lieg/\liegk)^\vee\otimes 
A_{\lieq}(\mu)\otimes M(\mu)^\vee)_{\QQ(\mu)},
$$
with
$$
\omega_p\;\in\;\bigwedge^d(\lieg/\liegk)_{\QQ(\mu)}^\vee,\quad
a_p\;\in\;A_{\lieq}(\mu)_{\QQ(\mu)},\quad
m_p\;\in\;M(\mu)_{\QQ(\mu)}^\vee,\quad 1\leq p\leq s.
$$
Identifying $Z^{\rm s}$ with the maximal $\QQ$-split torus in the center of $G$, we introduce the group
$$
M\;:=\;H\cap Z^{\rm s}\;=\;Z^{\rm s},
$$
and denote its Lie algebra by $\liem$. The transpose of the diagonal embedding
$$
(\lieh/(\liel+\liem))_{\QQ(\mu)}\;\to\;(\lieg/\liegk)_{\QQ(\mu)},
$$
is a projection
$$
(\lieg/\liegk)_{\QQ(\mu)}^\vee\;\to\;
(\lieh/(\liel+\liem))_{\QQ(\mu)}^\vee.
$$
Together with the $d$-th exterior power of we obtain the map
$$
\res^G_H:\quad \bigwedge^d(\lieg/\liegk)_{\QQ(\mu)}^\vee\;\to\;
\bigwedge^d(\lieh/(\liel+\liem))_{\QQ(\mu)}^\vee,
$$
where the right hand side is one-dimensional due to the numerical coincidence
$$
d\;=\;
\dim\lieh/(\liel+\liem).\\
$$
We fix a $\QQ(\mu)$-rational basis vector
$$
0\neq w_0\;\in\;
\bigwedge^d(\lieh/(\liel+\liem))_{\QQ(\mu)}.
$$
The choice of $w_0$ is equivalent to the choice of an isomorphism of vector spaces
$$
\bigwedge^d(\lieh/\liel+\liem)_{\QQ(\mu)}^\vee\;\to\;\QQ(\mu),\quad
\omega\;\mapsto\;\omega(w_0).
$$
The left hand side is a one-dimensional $L$-module via the adjoint action on
$\lieh/(\liel+\liem)$, and we furnish the right hand side with an action of $L_1$ such that the above map becomes $L_1$-linear. The resulting one-dimensional $L_1$-module is denoted $\mathcal L$. It is of finite order, more concretely $\mathcal L^{\otimes 2}\;\cong\;{\bf 1}$.

Let us assume that the character $\mathcal N^{\otimes k}$ is critical for $M(\mu)^\vee$, and fix a non-zero element
$$
0\neq\xi_k\;\in\;\Hom_H(M(\mu)^\vee,\mathcal N^{\otimes k})_{\QQ(\mu)}.
$$
By \cite[Proposition 6.3.1, see also section 6.6]{grobnerraghuram2014}, we know that then all finite order twists
$$
\chi\;=\;\mathcal N^{\otimes k}\otimes\sgn_{\infty}^\delta,\quad\delta\in\{0,1\}^{r_F}
$$
are critical quasi-characters of $H(\RR)$. Now for each such $\chi$ and any
$$
\lambda\;\in\;\Hom_{\lieh,C}(A_{\lieq}(\mu),\chi_1^{-1})_{\QQ(\mu)},
$$
the $\QQ(\mu)$-rational functionals $\lambda$ and $\xi_k$ induce a $\QQ(\mu)$-rational $\pi_0(L_1)$-equivariant map
$$
\begin{CD}
H(\lambda\otimes\xi_k):\quad
H^d(\lieg,GK^0;A_{\lieq}(\mu)\otimes M(\mu)^\vee)_{\QQ(\mu)}@>{\lambda\otimes\xi_k}>>
H^d(\lieh,ML^0;\chi_1^{-1}[k])_{\QQ(\mu)}.
\end{CD}
$$
By Poincar\'e duality, our choice of vector $w_0$ induces an isomorphism
$$
H^d(\lieh,ML^0;\chi_1^{-1}[k])_{\QQ(\mu)}\;\to\;
(\mathcal L\otimes\sgn_{\infty}^\delta)_{\QQ(\mu)},
$$
of $\pi_0(L_1)$-modules, because $\chi\cong\chi_1$ as $L_1$-modules. The composition of the latter with $H(\lambda\otimes\xi_k)$, provides us with a $\pi_0(L_1)$-equivariant map
$$
I(\lambda\otimes\xi_k):\quad
H^d(\lieg,GK^0;A_{\lieq}(\mu)\otimes M(\mu)^\vee)_{\QQ(\mu)}\;\to\;
(\mathcal L\otimes\sgn_{\infty}^\delta)_{\QQ(\mu)},
$$
which on the level of complexes is given explicitly by
$$
h=\sum_{p=1}^s\omega_p\otimes a_p\otimes m_p\;\mapsto\;\sum_{p=1}^s\omega_p(w_0)\otimes \lambda(a_p)\otimes \xi_k(m_p).
$$

\subsection{The global period relation}

In this section we prove
\begin{theorem}\label{thm:globalshalikarelation}
Assume $n\geq 1$, let $F$ be a totally real number field, and $\Pi$ a regular algebraic irreducible cuspidal automorphic representation of $\GL_{2n}(\Adeles_F)$ admitting a Shalika model. Assume that $\Pi_\infty$ has non-trivial Lie algebra cohomology with coefficients in an irreducible rational $G_{2n}$-module $M$, which we assume to be critical. Denote by $s_0=\frac{1}{2}+j_0$ the left most critical value of the standard $L$-function $L(s,\Pi)$ (such a $s_0$ exists). Then there exist non-zero periods $\Omega_\pm$, numbered by the characters $\pm$ of $\pi_0((F\otimes_\QQ\RR)^\times)$, such that for each critical half integer $s_1=\frac{1}{2}+j_1$ for $L(s,\Pi)$, and each finite order Hecke character
$$
\chi:\quad F^\times\backslash\Adeles_F^\times\to\CC
$$
we have
$$
\frac{L(s_1,\Pi\otimes\chi)}
{G(\overline{\chi})^{n}\Omega_{(-1)^{j_1}\sgn\chi}}
\;\in\;
\frac{L(s_0,\Pi_\infty)}
{L(s_1,\Pi_\infty)}\cdot
i^{j_1r_Fn}
\QQ(\Pi,\chi).
$$
Furthermore, for every $\tau\in\Aut(\CC/\QQ)$,
$$
i^{-j_1r_Fn}
\frac{L(s_0,\Pi_\infty)}
{L(s_1,\Pi_\infty)}\cdot
\frac{L(s_1,\Pi^\tau\otimes\chi^\tau)}
{G(\overline{\chi^\tau})^{n}\Omega_{(-1)^{j_1}\sgn\chi^\tau}}
\;=\quad\quad
$$
$$
\quad\quad
\left(
i^{-j_1r_Fn}
\frac{L(s_1,\Pi_\infty)}
{L(s_0,\Pi_\infty)}\cdot
\frac{L(s_1,\Pi\otimes\chi)}
{G(\overline{\chi})^{n}\Omega_{(-1)^{j_1}\sgn\chi}}
\right)^\tau
$$
\end{theorem}

\begin{proof}
We proceed as in section 4.6 of \cite{januszewskipart1}. Let $\chi=\otimes_v\chi_v$ be an algebraic Hecke character of $F$ with
$$
\chi_\infty\;=\;\mathcal N^{\otimes k(\chi)}\otimes\sgn_\infty^{\delta(\chi)}
$$
critical for $\Pi_\infty$. The period $\Omega(\chi_\infty)\in\CC^\times$ in the global special value formula
$$
\frac{\Lambda(\frac{1}{2},\Pi\otimes\chi)}
{G(\chi)^{m}\cdot
\Omega(\chi_\infty)}
\;\in\;
\QQ(\mu,\chi)
$$
(cf.\ Theorem 7.2.1 in \cite{grobnerraghuram2014}), arises as follows. We fix for each signature $\delta\in\{0,1\}^{r_F}$ a generator
$$
h_{\delta}
\;\in\;H^d(\lieg,GK^0;A_{\lieq}(\mu)\otimes M(\mu)^\vee)_{\QQ(\mu)}\;=\;\QQ(\mu)[\pi_0(L_1)]
$$
of the generalized $\sgn_\infty^{\delta}$-eigenspace for the action of $\pi_0(L_1)=\pi_0(K_n)$. We specialize to the signature $\delta=\delta(\chi)$ satisfying the compatibility condition
\begin{equation}
\sgn_\infty^{\delta}\;=\;\mathcal L\otimes \sgn_\infty^{\delta(\chi)+k(\chi)}.
\label{eq:chidelta}
\end{equation}
Now together with a choice of factorizable test vector in the Shalika model
$$
t^{(\infty)}\;=\;\otimes_v t_v\;\in\;
\mathscr S(\Pi^{(\infty)},\omega^{(\infty)},\psi^{(\infty)}),
$$
at the finite places $v\nmid\infty$, the cohomological vector $h_{\delta}$ gives rise to a cohomology class
$$
t_{\delta}\;:=\;
\sum_{p=1}^s \omega_p\otimes (\iota_\infty(a_p)\otimes t^{(\infty)})\otimes m_p\;\in\;
H^d(\lieg,GK^0;\mathscr S(\Pi,\eta,\psi)^{(K)}\otimes M(\mu)^\vee)_{\CC}.
$$
The inverse intertwining operator for the Shalika model turns this class into an automorphic cohomology class
$$
\vartheta_{\delta}\;:=\;
\sum_{p=1}^s \omega_p\otimes \Theta(\iota_\infty(a_p)\otimes t^{(\infty)})\otimes m_p\;\in\;
H^d(\lieg,GK^0;\Pi^{(K)}\otimes M(\mu)^\vee)_{\CC}.
$$
From this class we obtain a global cohomology class
$$
c(t_{\delta})\;\in\;H_{\rm c}^d(G(\QQ)\backslash G(\Adeles)/GK(\RR)^0K^{(\infty)};\underline{M(\mu)^\vee}_\CC),
$$
with coefficients in the local system associated to $M(\mu)^\vee$, and a suitable compact open $K^{(\infty)}$ which is small enough that $t'$ is $K^{(\infty)}$-invariant and additonally the underlying orbifold is a manifold.

By \cite{clozel1990} we know that $\Pi^{(\infty)}$ is defined over its field of rationality $\QQ(\Pi)$. Under the assumption that $t_{\delta}$ is chosen in the natural $\QQ(\Pi)$-rational structure of $\mathscr S(\Pi^{(\infty)},\omega^{(\infty)},\Psi^{(\infty)})$ (cf.\ paragraphs 3.6-3.9 in \cite{grobnerraghuram2014}), we may renormalize $c(t')$ via a scalar $\Omega(t_{\delta})\in\CC^\times$, such that
\begin{equation}
\Omega(t_{\delta})\cdot c(t_{\delta})\;\in\;H_{\rm c}^d(G(\QQ)\backslash G(\Adeles)/GK(\RR)^0K^{(\infty)};\underline{M(\mu)}_{\QQ(\Pi)}).
\label{eq:normalizedrationalclass}
\end{equation}

To each algebraic Hecke character $\chi$ as above, that we interpret as a character of $H_1(\Adeles)$ via composition with the determinant, we may associate a cohomology class as follows. The character $(k(\chi),-k(\chi))$ of $Z_{\rm s}$ gives rise to a rational character of $H$ that we denote again $(k(\chi),-k(\chi))$. Then we may attach to $\chi$ a rational cohomology class
$$
c_\chi\;\in\;
\begin{cases}
H^0(H(\QQ)\backslash H(\Adeles)/L(\RR)^0L^{(\infty)}(\chi);\underline{(k(\chi))}_{\QQ(\chi)}),\\
H^0(H(\QQ)\backslash H(\Adeles)/ML(\RR)^0L^{(\infty)}(\chi);\underline{(k(\chi),-k(\chi))}_{\QQ(\chi)}),\\
\end{cases}
$$
where $K^{(\infty)}(\chi)\subseteq K^{(\infty)}$ is a compact open such that the finite part of the character $\chi$ factors over $\det(L^{(\infty)}(\chi))$.

Recall that by \eqref{eq:etainfinity} the character $(1\otimes\eta_\infty)$ may be identified with a rational cohomology class
$$
c_\eta'\;\in\;
H^0(H(\QQ)\backslash H(\Adeles)/ML(\RR)^0L^{(\infty)}(\eta);\underline{(0,w_1)}_{\QQ(\eta)}).
$$

Now the natural map
$$
H_{\rm c}^d(G(\QQ)\backslash G(\Adeles)/GK(\RR)^0K^{(\infty)};\underline{M(\mu)^\vee}_{\QQ(\Pi)})\;\to\;
$$
$$
H_{\rm c}^d(H(\QQ)\backslash H(\Adeles)/ML(\RR)^0L^{(\infty)};\underline{M(\mu)^\vee}_{\QQ(\Pi)})\;\to\;
$$
$$
\begin{CD}
@>(-)\cup c_{\eta^{-1}}'\cup c_\chi>>
H_{\rm c}^d(H(\QQ)\backslash H(\Adeles)/ML(\RR)^0L^{(\infty)};\underline{M(\mu)^\vee\otimes(k(\chi),-w_1-k(\chi))}_{\QQ(\Pi,\eta,\chi)})
\end{CD}
$$
together with Poincar\'e duality for the right hand side, induces the modular symbol
$$
H_{\rm c}^d(G(\QQ)\backslash G(\Adeles)/GK(\RR)^0K^{(\infty)};\underline{M(\mu)^\vee}_{\QQ(\Pi)})\;\to\;
H^0(\Gamma; M(\mu)^\vee\otimes(k(\chi),-w_1-k(\chi))_{\QQ(\Pi)}),
$$
where $\Gamma\subseteq H(\QQ)$ is the arithmetic subgroup corresponding to
$$
L^{(\infty)}\;=\;H(\Adeles^{(\infty)})\cap K^{(\infty)}.
$$
Composition of the modular symbol with $\xi_k$ provides us with a $\pi_0(L_1)$-equivariant map
$$
H_{\rm c}^d(G(\QQ)\backslash G(\Adeles)/GK(\RR)^0K^{(\infty)};\underline{M(\mu)}_{\QQ(\Pi)})\;\to\;
(\mathcal N^{\otimes k+k(\chi)})_{\QQ(\Pi)}.
$$
The image of $\Omega(t_\delta)^{-1}\cdot c(t_{\delta})$ under this map computes the global integral
$$
\Omega(t_\delta)^{-1}\cdot
\sum_{p=1}^r
\omega_p(w_0)\cdot I(\frac{1}{2}+k,\chi,\Theta(\iota_\infty(a_p)\otimes t^{(\infty)}))\cdot\xi_k(m_p)\;\in\;\mathcal N^{\otimes k}_\CC.
$$
This value can only be non-zero if the compatibility condition \eqref{eq:chidelta} is satisfied, i.e.
$$
\sgn_\infty^{\delta+k}\;=\;\mathcal L\otimes \sgn_\infty^{\delta(\chi)+k(\chi)}.
$$
By Grobner and Raghuram \cite[section 3.9 and Proposition 5.2.3]{grobnerraghuram2014} we know that we may find a $\QQ(\pi,\eta,\chi)$-rational $t^{(\infty)}\in\mathscr S(\Pi^{(\infty)},\eta^{(\infty)},\psi^{(\infty)})$, a choice is equivariant under the action of $\Aut(\CC/\QQ)$, such that
$$
e^{(\infty)}(\chi^{(\infty)}|\cdot|_{\Adeles^{(\infty)}}^s,t^{(\infty)})\;=\;G(\chi)^{n}
L(s,\Pi^{(\infty)}\otimes\chi^{(\infty)}).
$$
Therefore the image of the cohomology class $c(t_\delta)$ under the modular symbol composed with $\xi_k$ computes the value
$$
G(\chi)^{n}\cdot
\Lambda(\frac{1}{2}+k,\Pi\otimes\chi)\cdot
\sum_{p=1}^r
\omega_p(w_0)\cdot e_\infty(\frac{1}{2}+k,\chi_\infty,\iota_\infty(a_p))\cdot\xi_k(m_p).
$$
Since the values $a_p$ lie in the same $\QQ(\mu)$-rational subspace as the good test vector $t$,  we conclude the proof of the desired period relation with by Corollary \ref{cor:criticalrelation}.
\end{proof}

\section{Deligne's Conjecture}\label{sec:deligne}

In this section we discuss the relation between Theorem \ref{thm:globalshalikarelation} and Deligne's Conjecture on special values of $L$-functions \cite{deligne1979}.

\subsection{Motives and Hodge numbers}

Attached to $\Pi$ is a conjectural motive $M(\Pi)$ over $F$ which is characterized by the conjectural identity of $L$-functions (cf.\ \cite{clozel1990})
\begin{equation}
L(s-\frac{2n-1}{2},\Pi)\;=\;L(s,M(\Pi)).
\label{eq:shalikalidentity}
\end{equation}
We know that we may attach a compatible system of $\ell$-adic Galois representations
$$
\rho_{\Pi,\ell}:\Gal(\overline{F}/F)\to\GL_{2n}(\overline{\QQ}_\ell)
$$
to $\Pi$ for $\ell$ outside a finite set of exceptional primes. These conjecturally provide us with the $\ell$-adic realizations of the conjectural motive $M(\Pi)$. In terms of this compatible system the identity \eqref{eq:shalikalidentity} holds outside a finite set of exceptional primes and infinity, in a suitable right half plane.

To each archimedean place $v$ of $F$ we may associate an embedding $v:F\to \CC$, which gives rise to a base change $M(\Pi)\times_{F,v}\CC$ which is a motive over $\CC$. This complex motive admits a Betti realization,
$$
H_{{\rm B},v}(M(\Pi))\;:=\;H_{{\rm B}}(M(\Pi)\times_{F,v}\CC),
$$
which is a finite-dimensional $\QQ$-vector space of dimension $d:=2n$. Hodge theory provides us with a bigraduation $\{H_v^{p,q}\}_{p,q\in\ZZ}$ of the complexification
$$
H_{{\rm B},v}(M(\Pi))_\CC\;:=\;
H_{{\rm B},v}(M(\Pi))\otimes_\QQ\CC.
$$
The numbers
$$
h_v^{p,q}\;:=\;\dim H_v^{p,q}
$$
are the Hodge numbers of $M(\Pi)$ relative to $v$.

Complex conjugation induces an involution
$$
F_{\infty,v}:\quad H_{{\rm B},v}(M(\Pi))\to H_{{\rm B},v}(M(\Pi)).
$$
It interchanges $H_v^{p,q}$ with $H_v^{q,p}$, and in particular $h_v^{p,q}=h_v^{q,p}$. It is believed that $M(\Pi)$ is pure of a fixed weight $w\in\ZZ$, which means that
$$
h_v^{p,q}\;\neq\;0\quad\Longrightarrow\quad p+q=w,
$$
for a $w$ independent of $v$.

More concretely the conjectural Hodge numbers $h_v^{(p,q)}$ for each archimedean place $v$ of $F$ are expected to vanish except in the following cases:
$$
(p,q)\;=\;(\mu_{k,v}-k+2n,\mu_{2n+1-k,v}+k-1),\quad 1\leq k\leq 2n.
$$
In these cases they are conjecturally given by
$$
h_v^{p,q}\;=\;1.
$$

The action of $F_{\infty,v}$ decomposes
$$
H_{{\rm B},v}(M(\Pi))\;=\;H_{{\rm B},v}^+(M(\Pi))\oplus H_{{\rm B},v}^-(M(\Pi))
$$
into $(\pm1)$-eigenspaces. We let
$$
d^{\pm}_v\;:=\;\dim H_{{\rm B},v}^\pm(M(\Pi))\;=\;n.
$$

The de Rham realization of $M(\Pi)$ is an $F$-vectorspace $H_{\rm dR}(M(\Pi))$ of dimension $2n$, and comes with a decreasing filtration $F_{\rm dR}^{p}$, given by the hyper cohomology spectral sequence associated to the (motivic) de Rham complex. We write
$$
H_{{\rm dR},v}(M(\Pi))_\CC\;:=\;
H_{{\rm dR}}(M(\Pi))\otimes_{v,F}\CC,
$$
and denote the induced filtration by $F_{{\rm dR},v}^p$.

By GAGA we have for each archimedean place $v$ of $F$ a comparison isomorphism
$$
\begin{CD}
\iota_v:\quad
H_{{\rm B},v}(M(\Pi))_\CC @>\sim>>\;H_{{\rm dR},v}(M(\Pi))_\CC,
\end{CD}
$$
which respects the Hodge decomposition resp.\ the Hodge filtration, i.e.
$$
\iota_v\left(\bigoplus_{p'\geq p}H_v^{p',q'}\right)\;=\;F_{{\rm dR},v}^p.
$$
Following Deligne we define
$$
H_{\rm dR}^\pm(M(\Pi,\sigma))_{v,\CC}\;:=\;H_{\rm dR}(M(\Pi))_\CC/F_v^{\mp},
$$
where
$$
F_v^{\pm}\;:=\;F_{{\rm dR},v}^{\frac{w+1}{2}}.
$$
By dimension counting, we see that we have the following refined comparison isomorphisms
$$
\begin{CD}
\iota_v^\pm:\quad
H_{{\rm B},v}^\pm(M(\Pi))_\CC @>\sim>>\;H_{{\rm dR},v}^\pm(M(\Pi))_\CC.
\end{CD}
$$
As the left hand side has a natural $\QQ$-structure, and the right hand side a natural $F$-structure, Deligne chooses basis in these rational structures and defines the periods
$$
c^\pm_v(M(\Pi))\;:=\;\det(\iota_v^\pm)
$$
with respect to these bases. Then $c^\pm_v(M(\Pi))$ is uniquely defined up to scalars in $F^\times$. We set for each finite order character $\varepsilon$ of
$$
\pi_0((F\otimes_\QQ\RR)^\times)\;=\;\pi_0(L),
$$
\begin{equation}
c_\varepsilon\;=\;\prod_v c_v^{\epsilon({\bf1}_v)}(M(\Pi)),
\label{eq:periodproduct}
\end{equation}
and
\begin{equation}
d^\varepsilon\;=\;\sum_v d_v^{\epsilon({\bf1}_v)}\;=\;r_Fn,
\label{eq:shalikadcalc}
\end{equation}
independently of the signature $\varepsilon$.

By the calculation in section 6.1 in \cite{grobnerraghuram2014} we obtain
\begin{equation}
\frac{L(s_1,\pi_\infty\times\sigma_\infty)}
{L(s_0,\pi_\infty\times\sigma_\infty)}\;\in\;(2\pi)^{(s_0-s_1)r_Fn}\QQ^\times.
\label{eq:shalikapicalc}
\end{equation}

Deligne's conjecture for the special values of $L(s,M(\Pi)\otimes\chi)$ reads

\begin{conjecture}[Deligne]\label{conj:shalikadeligne}
For each $k\in\ZZ$ critical for $L(s,M(\Pi))$ and each finite order character $\chi:\Gal(\overline{F}/F)\to\CC^\times$ we have
$$
\frac{L(k,M(\Pi)\otimes\chi)}
{G(\overline{\chi})^{n}(2\pi i)^{kr_Fn}c_{(-1)^k\sgn\chi}}\;\in\;\QQ(\Pi,\chi).
$$
\end{conjecture}

\subsection{Compatibility with Deligne's conjecture}

In the context of the previous section Theorem \ref{thm:globalshalikarelation} gives

\begin{theorem}\label{thm:globalshalikaperiods}
Assume $n\geq 1$, let $F$ be a totally real number field, and $\Pi$ be an algebraic irreducible cuspidal automorphic representation of $\GL_{2n}(\Adeles_F)$ admitting a Shalika model. Assume that $\Pi_\infty$ has non-trivial Lie algebra cohomology with coefficients in an irreducible rational $G_{2n}$-module $M$, which we assume to be critical. Then there exist non-zero periods $\Omega_\pm$, numbered by the $2^{r_F}$ characters $\pm$ of $\pi_0((F\otimes_\QQ\RR)^\times)$, such that for each critical half integer $s_0=\frac{1}{2}+j_0$ for $L(s,\Pi)$, and each finite order Hecke character
$$
\chi:\quad F^\times\backslash\Adeles_F^\times\to\CC^\times
$$
we have, in accordance with Deligne's Conjecture \ref{conj:shalikadeligne},
$$
\frac{L(s_0,\Pi\otimes\chi)}
{G(\overline{\chi})^{n}(2\pi i)^{j_0r_Fn}
\Omega_{(-1)^{j_0}\sgn\chi}}
\;\in\;
\QQ(\Pi,\chi).
$$
Furthermore this expression is equivariant under $\Aut(\CC/\QQ)$.
\end{theorem}

\begin{corollary}\label{cor:shalikadeligneverification}
Under the assumptions of Theorem \ref{thm:globalshalikaperiods} Deligne's Conjecture \ref{conj:shalikadeligne} for the conjectural motive $M(\Pi)\otimes\chi$, for an Artin motive $\chi$ of rank $1$, is equivalent to the statement
$$
\frac{
\Omega_{(-1)^{j}\sgn\chi}}
{(2\pi i)^{r_Fn^2}c_{(-1)^{j+n}\sgn\chi}}
\;\in\;\QQ(\Pi)^\times,\quad j\in\{0,1\}.
$$
\end{corollary}

\bibliographystyle{plain}

\end{document}